\documentclass[11pt]{amsart}

\usepackage[a4paper,hmargin=3.5cm,vmargin=4cm]{geometry}
\usepackage{amsfonts,amssymb,amscd,amstext,verbatim}

\usepackage{graphicx}
\usepackage[dvips]{epsfig}

\usepackage{fancyhdr}
\usepackage[utf8]{inputenc}
\usepackage{hyperref}
\usepackage{verbatim}

\usepackage{times}

\usepackage{enumerate}
\usepackage{titlesec}
\usepackage{mathrsfs}

\titleformat{\section}
{\filcenter\bfseries\large} {\thesection{.}}{0.2cm}{}
\titleformat{\subsection}[runin]
{\bfseries} {\thesubsection{.}}{0.15cm}{}[.]
\titleformat{\subsubsection}[runin]
{\em}{\thesubsubsection{.}}{0.15cm}{}[.]

\newtheorem{theorem}{Theorem}[section]
\newtheorem{proposition}[theorem]{Proposition}

\newtheorem{lemma}[theorem]{Lemma}
\newtheorem{corollary}[theorem]{Corollary}

\theoremstyle{definition}
\newtheorem{definition}[theorem]{Definition}
\newtheorem{remark}[theorem]{Remark}

\newtheorem{problem}[theorem]{Problem}

\numberwithin{equation}{section}
\numberwithin{figure}{section}

\newcommand\Ccal{\mathcal{C}}

\newcommand\Ocal{\mathcal{O}}

\newcommand\Ascr{\mathscr{A}}

\newcommand\Cscr{\mathscr{C}}

\newcommand\Escr{\mathscr{E}}

\newcommand\Lscr{\mathscr{L}}

\newcommand\Oscr{\mathscr{O}}
\newcommand\Pscr{\mathscr{P}}

\newcommand\B{\mathbb{B}}
\newcommand\C{\mathbb{C}}
\newcommand\D{\overline{\mathbb D}}
\newcommand\CP{\mathbb{CP}}
\renewcommand\D{\mathbb D}

\renewcommand\H{\mathbb{H}}

\newcommand\N{\mathbb{N}}
\renewcommand\P{\mathbb{P}}
\newcommand\R{\mathbb{R}}

\renewcommand\S{\mathbb{S}}

\newcommand\Z{\mathbb{Z}}

\newcommand\igot{\mathfrak{i}}

\renewcommand\igot{\mathfrak{i}}

\renewcommand\imath{\igot}

\newcommand\hra{\hookrightarrow}
\newcommand\lra{\longrightarrow}

\newcommand\longhookrightarrow{\ensuremath{\lhook\joinrel\relbar\joinrel\rightarrow}}

\newcommand\wt{\widetilde}

\newcommand\di{\partial}
\newcommand\dibar{\overline\partial}

\newcommand\dist{\mathrm{dist}}
\renewcommand\span{\mathrm{span}}

\newcommand\cd{\overline{\D}}

\newcommand\Id{\mathrm{Id}}

\def\dist{\mathrm{dist}}
\def\span{\mathrm{span}}

\newcommand{\bea}{\begin{eqnarray*}}
\newcommand{\eea}{\end{eqnarray*}}

\numberwithin{equation}{section}

\begin{document}

\fancyhead[LO]{Mergelyan approximation theorem for holomorphic Legendrian curves}
\fancyhead[RE]{F.\ Forstneri{\v c}} 
\fancyhead[RO,LE]{\thepage}

\thispagestyle{empty}

\vspace*{1cm}
\begin{center}
{\bf\LARGE Mergelyan approximation theorem for holomorphic \\ Legendrian curves}

\vspace*{0.5cm}

{\large\bf  Franc Forstneri{\v c}} 
\end{center}

\vspace*{1cm}

\begin{quote}
{\small
\noindent {\bf Abstract}\hspace*{0.1cm}
In this paper, we prove a Mergelyan-type approximation theorem for immersed holomorphic Legendrian curves
in an arbitrary complex contact manifold $(X,\xi)$. Explicitly, we show that if $S$ is a compact admissible
set  in a Riemann surface $M$ and $f:S\to X$ is a $\xi$-Legendrian immersion of class $\Cscr^{r+2}(S,X)$ 
for some $r\ge 2$ which is holomorphic in the interior of $S$, then $f$ can be approximated 
in the $\Cscr^r(S,X)$ topology by holomorphic Legendrian embeddings from open neighbourhoods 
of $S$ into $X$. This has numerous applications, some of which are indicated in the paper.
In particular, by using Bryant's correspondence for the Penrose twistor map $\CP^3\to S^4$
we show that a Mergelyan approximation theorem and the Calabi-Yau property
hold for conformal superminimal surfaces in the $4$-sphere $S^4$.

\vspace*{0.2cm}

\noindent{\bf Keywords}\hspace*{0.1cm}  complex contact manifold, Legendrian curve, Mergelyan theorem,
superminimal surface

\vspace*{0.1cm}

\noindent{\bf MSC (2020):}\hspace*{0.1cm}} 
Primary 53D35, 32E30, 34D10. Secondary 37J55, 53A10.
%
%
%
\end{quote}

%
%
%
%
\section{Introduction}\label{sec:intro}

A {\em complex contact manifold} is a pair $(X,\xi)$, where $X$ is a complex manifold of 
necessarily odd dimension $2n+1\ge 3$ and $\xi$ is a {\em complex contact subbundle} of 
the holomorphic tangent bundle $TX$, that is, a maximally nonintegrable
holomorphic hyperplane subbundle  of $TX$.  More precisely, the {\em O'Neill tensor} 
$\xi\times \xi \to TX/\xi=L$, $(v,w)\mapsto [v,w] \mod \xi$, is nondegenerate. 
Consider the short exact sequence
\begin{equation}\label{eq:eta}
	0 \lra \xi \longhookrightarrow TX \stackrel{\eta}{\lra} L \lra 0,
\end{equation}
where $\eta$  is a holomorphic $1$-form on $X$ with values in the 
line bundle $L=TX/\xi$ realising the quotient projection, so $\xi =\ker\eta$. 
The contact condition is equivalent to $\eta\wedge (d\eta)^n$ having no zeros on $X$.
A complex contact structure is locally in a neighbourhood of any point of $X$ 
equivalent to the standard contact structure on $\C^{2n+1}$ given by the $1$-form
\begin{equation}\label{eq:standard}
	\eta_{\mathrm{std}} = dz+\sum_{j=1}^n x_j dy_j.
\end{equation}
See Darboux \cite{Darboux1882CRAS}, Moser \cite{Moser1965TAMS}, and 
Geiges  \cite[p.\ 67]{Geiges2008} for the real case, 
and Alarc\'on et al. \cite[Theorem A.2]{AlarconForstnericLopez2017CM} for the holomorphic case. 
More complete introductions to complex contact manifolds are available in the papers
by LeBrun \cite{LeBrun1995IJM}, LeBrun and Salamon \cite{LeBrunSalamon1994IM},
Beauville \cite{Beauville2011}, and Alarc\'on and Forstneri\v c \cite{AlarconForstneric2019IMRN}.

A map $f: M\to X$ of class $\Cscr^1$ from a real or complex manifold $M$ is said to be 
{\em isotropic}, or an {\em integral submanifold} of the contact structure $\xi$ on $X$, if 
\begin{equation}\label{eq:isotropic}
	df_x (T_x M)\subset \xi_{f(x)}\ \ \text{holds for all}\ \ x\in M.
\end{equation}
Equivalently, $f^*\eta=0$ for any $1$-form $\eta$ on $X$ with $\ker\eta=\xi$.
If $f$ is an immersion at a generic point of $M$, then $\dim_\R M\le 2n$ where $\dim_\C X=2n+1$. 
Isotropic submanifolds of maximal dimension $2n$ are necessarily complex 
submanifolds of $X$ (see \cite[Lemma 5.1]{AlarconForstneric2019IMRN});
they are called {\em Legendrian}. In this paper, isotropic holomorphic curves 
from Riemann surfaces are called {\em holomorphic Legendrian curves}
irrespectively of the dimension of $X$. 

We shall be considering Legendrian immersions from sets of the following type.
 
%
%
\begin{definition}[Admissible sets] \label{def:admissible}
Let $M$ be a connected Riemann surface. A proper compact subset $S\subsetneq M$ is {\em admissible} 
if  it is of the form $S=K\cup E$, where $K$ is a union of finitely many pairwise disjoint compact domains 
in $M$ with piecewise smooth boundaries and $E = \overline{S \setminus K}$ is a union 
of finitely many pairwise disjoint smooth Jordan arcs and closed Jordan curves 
meeting $K$ only in their endpoints (if at all) and such that their intersections with the boundary
$bK$ of $K$ are transverse.
\end{definition} 

An admissible set has a well-defined tangent space at every point $x\in S$
which equals $T_xM$ when $x\in K$, and equals the tangent space of the curve 
$E=\overline{S\setminus K}$ at any point $x\in S\setminus K$. Given a complex manifold $X$, 
we denote by $\Ascr^r(S,X)$ the space of maps $S\to X$ of class $\Cscr^r$ which are holomorphic
in the interior $\mathring S$. Every $f\in \Ascr^r(S)=\Ascr^r(S,\C)$ can be approximated in $\Cscr^r(S)$ 
by meromorphic functions on $M$ with poles off $S$, and every map $S\to X$ of class $\Ascr^r(S,X)$ 
into an arbitrary complex manifold $X$ can be approximated in the $\Cscr^r(S,X)$ topology by 
holomorphic maps $U\to X$ from open neighbourhoods $U\subset M$ of $S$ 
(see \cite[Theorem 16 and Corollary 9]{FornaessForstnericWold2018}).

A map $f\in \Ascr^1(S,X)$ from an admissible set $S$ into a complex contact manifold $(X,\xi)$ is 
said to be {\em Legendrian} if condition \eqref{eq:isotropic} holds with $M$ replaced by $S$.

Our  main result is the following Mergelyan-type approximation theorem
for holomorphic Legendrian immersions.

%
%
\begin{theorem}
\label{th:Mergelyan}
Let $(X,\xi)$ be a complex contact manifold, and let $S$ be an admissible set in a Riemann surface $M$.
Then, every Legendrian immersion $f:S\to X$ of class $\Ascr^{r+2}(S,X)$ for $r\ge 2$ can be approximated 
in the $\Cscr^r(S,X)$ topology by holomorphic Legendrian immersions $\tilde f:U\to X$
from open neighbourhoods of $S$ in $M$. Furthermore, $\tilde f$ can be chosen to agree with $f$ 
on any given finite subset $A$ of $S$ (to any given finite order at the points of $A\cap \mathring S$),
and it can be chosen an embedding provided that  $f$ is injective on $A$. 
\end{theorem}

%
%
The loss of two derivatives is likely due to the method; there is 
on loss in the standard contact structure on $\C^{2n+1}$ (cf.\ \cite{AlarconForstnericLopez2017CM}). 
Theorem \ref{th:Mergelyan} is proved in Section \ref{sec:proof}; here is an outline. 
We have $\xi=\ker\eta$ where $\eta$ is a holomorphic contact form (see \eqref{eq:eta}).
If $B$ is a domain in $\C^n$, we denote by $\Ascr^r(S\times B, X)$ the space of  
maps $F:S\times B\to X$ of class $\Cscr^r$ which are holomorphic on $\mathring S\times B$ 
and such that $F(x,\cdotp):B\to X$ is holomorphic for every $x\in S$. 
The immersion $f:S\to X$ extends to an immersion $F:S\times \B^{2n}\to X$ 
of class $\Ascr^{r+2}(S\times \B^{2n}, X)$, where $\B^{k}$ denotes the open unit ball in $\C^{k}$ 
(see Lemma \ref{lem:normalbundle}). This gives a contact form $\beta=F^*\eta$ on $S\times \B^{2n}$ 
of class $\Ascr^{r+2}(S\times \B^{2n})$. Next, we obtain a partial normal form of 
$\beta$ along the Legendrian curve $S\times \{0\}^{2n}$
in the spirit of \cite[Theorem 1.1]{AlarconForstneric2019IMRN}; see Lemma \ref{lem:NF}. 
This amounts to precomposing $F$ by a fibre-preserving map on $S\times \B^{2n}$
of class $\Ascr^r$ which is $\C$-linear on the fibres, 
so the new immersion $F$ is of class $\Ascr^r(S\times \B^{2n})$ (here we lose two derivative). 
From this point on, the  proof amounts to a careful study of the period map 
of solutions of the differential equation for $\beta$-Legendrian curves along closed real curves in $S$ 
forming a homology basis. (This period map is also called the {\em monodromy map} when considered
on the whole space of solutions, or the {\em Poincar\'e first return map} 
when considered on a neighbourhood of a given single-valued solution.) 
The differential equation for Legendrian curves is underdetermined, and we obtain a family 
of determined differential equations by choosing all but one independent variables as functions of the remaining 
independent variable and of some additional parameters.   
We show that a suitable choice of such a family admits non-single-valued solutions with 
arbitrarily given small periods around each closed curve in a homology basis of $S$.
Next, we approximate $F$ in the $\Cscr^r(S\times \B^{2n}, X)$ topology by a holomorphic immersion 
$G:U\times \rho\B^{2n}\to X$, where $U\subset M$ is a neighbourhood of $S$ and $0<\rho<1$,
and consider the holomorphic contact form $\gamma=G^*\eta$ on $U\times \rho\B^{2n}$. 
If the approximation is close enough, the degree theory shows that there is a parameter value for 
which the associated differential equation in the perturbed family has a solution with vanishing periods 
on a neighbourhood of $S$. The $G$-image of this single-valued solution is an immersed holomorphic 
$\eta$-Legendrian curve $\tilde f:U\to X$ whose restriction to $S$ approximates the given Legendrian curve 
$f:S\to X$. Finally, by \cite[Theorem 1.2]{AlarconForstneric2019IMRN} we can approximate $\tilde f$ by 
embedded holomorphic Legendrian curves.

%
%
The restriction to admissible sets in Theorem \ref{th:Mergelyan}
is rather natural as one cannot reasonably talk about solutions of differential
equations on topologically complicated sets, such as those in the  approximation theorems 
for functions due to Mergelyan  \cite{Mergelyan1951} and Vitushkin \cite{Vitushkin1966}.
Our proof strongly uses the fact that $S$ 
has finitely generated first homology group $H_1(S,\Z)\cong \Z^l$, with a homology basis 
consisting of piecewise $\Cscr^1$ curves. Some generalisations are possible 
in the spirit of Carleman-type approximation; see Castro-Infantes and Chenoweth 
\cite{Castro-InfantesChenoweth2020} for results of this type for conformal minimal surfaces
in $\R^n$, as well as the survey \cite{FornaessForstnericWold2018}. 
We shall not pursue them here.

%
%
\begin{remark}\label{rem:Hilbert21}
The problem of finding (solutions of) ordinary holomorphic differential equations with prescribed 
periods or holonomy is related to Hilbert's 21st problem, asking about the existence
of systems of ordinary linear holomorphic differential equations whose coefficients are rational 
functions with simple poles at a given finite set of points in the extended complex plane 
$\C\cup\{\infty\}=\CP^1$ (Fuchsian systems) and with a given monodromy. 
After major contributions by Hilbert (1905) and Plemelj (1908), the problem was solved 
in the negative by Bolibruch in 1989; see \cite{AnosovBolibruch1994}. 
Plemelj \cite{Plemelj1908} gave an affirmative answer to Hilbert's 21st problem 
in the bigger class of systems whose coefficients have simple poles and the solutions grow 
at most polynomially at every pole. Here we deal with nonlinear Pfaffian holomorphic 
differential equations which are underdetermined and completely nonintegrable. 
This allows us to find a parametric family of determined equations which have solutions with 
arbitrarily given small periods on curves in a homology basis.
\qed\end{remark}

Admittedly, our proof is substantially more involved than the proof of the corresponding result in the seemingly 
only known special case when $X$ is a Euclidean space $\C^{2n+1}$ endowed with the standard complex 
contact structure \eqref{eq:standard}; see \cite[Theorems 1.1 and 5.1]{AlarconForstnericLopez2017CM}. 
Assuming that the admissible set $S\subset M$ is $\Ocal(M)$-convex (equivalently, $S$ has no holes in $M$), 
those results show that one can approximate a Legendrian map $S\to\C^{2n+1}$ of class 
$\Ascr^r(S,\C^{2n+1})$ for any $r\in\N$ in the $\Cscr^r(S,\C^{2n+1})$ topology 
by proper holomorphic Legendrian embeddings $M\hra \C^{2n+1}$. 
We do not see a comparably simple approach in more general complex contact manifolds.

The simplest example of an admissible set with a nontrivial homology 
is the circle $S^1=\{z\in\C:|z|=1\}$ in $M=\C$. 
In this case, Theorem \ref{th:Mergelyan} gives the following corollary.
Given $\rho>1$ we denote by $A_\rho$ the annulus $\{z\in\C: \rho^{-1}<|z|<\rho\}$. 

%
%
\begin{corollary}\label{cor:loops}
Let $X$ be a complex contact manifold. Every immersed Legendrian loop $S^1\to X$ of class 
$\Cscr^{r}(S^1,X)$ $(r\ge 1)$ can be approximated in $\Cscr^r(S^1,X)$ by embedded holomorphic 
Legendrian annuli $A_\rho\hra X$, where $\rho>1$ may depend on the map.
\end{corollary}

To see this, we first use well-known results in nonintegrable Pfaffian systems 
(see Gromov \cite[pp.\ 113--114]{Gromov1996}) in order to approximate the given immersion by a smooth 
isotropic immersion, and we apply our Theorem \ref{th:Mergelyan} to this new smooth curve.  
It is furthermore possible to uniformly approximate any closed Jordan curve in $X$ by a Legendrian
immersion. Corollary \ref{cor:loops} is equivalent to the statement that every immersed Legendrian
loop $S^1\to X$ in a complex contact manifold can be approximated by real analytic Legendrian loops, 
since the complexification of such are holomorphic Legendrian immersion from small surrounding annuli.

\begin{problem}\label{prob:realanalytic}
Let $X$ be a complex contact manifold and $f:M\to X$ be a smooth isotropic immersion from
a compact real analytic manifold $M$. Is it possible to approximate $f$ by real analytic 
isotropic immersions $M\to X$?
\end{problem}

The answer to this question is affirmative if $(X,\xi)$ is a real analytic contact manifold,
i.e., if $\xi$ is a real codimension one contact subundle of $TX$; see 
Cieliebak and Eliashberg \cite[Sect.\ 6.7]{CieliebakEliashberg2012}. 
However, their method does not apply in the holomorphic case, for two reasons.
One is that holomorphic approximation theory is very different from real analytic approximation
theory, the latter allowing arbitrarily close approximation (even in fine topologies) everywhere, and 
not only on the given subset. This is impossible in the holomorphic case
where the approximants necessarily diverge outside the set of approximation, unless the 
original map extends holomorphically to a bigger domain, in which case the problem is void. 
The second one is that the proof in \cite{CieliebakEliashberg2012} 
relies on Gray's stability theorem for real contact structures, which fails 
even locally for real analytic deformations of complex contact structures.
I wish to thank Roger Casals and Nicola Pia for the following information
(private communication, November 2019). 
The space of $k$-dimensional distributions on $\R^n$ (i.e., vector subbundles 
of rank $k$ of the tangent bundle $T\R^n$) has functional dimension $k(n-k)$,
the dimension of the Grassmanian manifold of $k$-planes in $\R^n$. 
The symmetry group of diffeomorphisms is given by $n$ functions, so
the functional dimension of $k$-distributions on $\R^n$ up to diffeomorphisms is $k(n-k)-n$. 
The only cases when this is non-positive (and there is a unique
local normal form) are $(k,n)\in \{(1,n),(n-1,n),(2,4)\}$, corresponding to vector fields, 
real contact and even-contact structures, and the Engel case.  
Any other pair $(k,n)$ produces distributions which have functional moduli and thus are not stable. 
More precise information can be found in Bryant et al. \cite[Sect.\ II.5]{BryantAl1991};
see in particular the examples in \cite[pp.\ 48--49]{BryantAl1991}.

%
%
Theorem \ref{th:Mergelyan} reduces the problem of approximating Legendrian 
immersions $S\to X$ from an admissible set $S$ in a Riemann surface $M$ by 
holomorphic Legendrian immersions $M\to X$ to the following

{\em Runge approximation problem for Legendrian immersions:} 
Let $S$ be a compact set in a Riemann surface $M$ and $f:U\to X$ be a holomorphic 
Legendrian immersion from an open neighbourhood of $S$ to a complex contact manifold $(X,\xi)$.
When is it possible to approximate $f$ on $S$ by holomorphic Legendrian immersions $M\to X$?

%
%
The answer depends on the pair $S\subset M$, as well as on the complex structure and 
the contact structure on $X$. It is negative in general even if $X=\C^3$ and $S$ is a closed disc in $\C$. 
Indeed, in \cite{Forstneric2017JGEA} the author gave an example of a complex
contact structure $\xi$ on $\C^3$ which is Kobayashi hyperbolic; 
in particular, there are no nonconstant holomorphic Legendrian curves $\C\to (\C^3,\xi)$, 
and there are holomorphic Legendrian discs $\overline \D\to (\C^3,\xi)$ which are not 
approximable by holomorphic Legendrian discs $2\D\to (\C^3,\xi)$.
Hence, Theorem \ref{th:Mergelyan} cannot be improved in general
even if $X$ is an Oka manifold (see \cite[Sect.\ 5.4]{Forstneric2017E} for this notion), 
for it is the properties of the contact structure on $X$ that matter as well.
Without the Legendrian condition, Runge approximation holds (in the absence of topological
obstructions) for maps from all Stein manifolds, in particular from all open Riemann surfaces, 
to any Oka manifold; see \cite[Theorem 5.4.4]{Forstneric2017E}.

%
%
We now describe an application of Theorem \ref{th:Mergelyan} to Legendrian curves in 
the complex projective space $\CP^3$.
Recall (see LeBrun and Salamon \cite{LeBrun1995IJM,LeBrunSalamon1994IM}) that 
$\CP^3$ admits an essentially unique complex contact structure, which is determined by the $1$-form
\[ 
	\eta=z_0dz_1-z_1dz_0+z_2dz_3-z_3dz_2 
\] 
on $\C^4$ via the standard projection $\C^4\setminus \{0\} \to \CP^3$.
On any affine chart $\C^3\subset \CP^3$ this gives the standard contact structure on $\C^3$.
The analogous statements hold on $\CP^{2n+1}$ for any $n\in\N$.
It was shown by Alarc\'on, Forstneri\v c and L\'arusson (see 
\cite[Theorems 3.2 and 3.4]{AlarconForstnericLarusson2019X}) 
that holomorphic Legendrian curves and immersions from Riemann surfaces (both open and closed)
into $\CP^3$ satisfy the Runge approximation theorem with interpolation, 
and every Legendrian immersion $M\to\CP^3$  from an open Riemann surface can be 
approximated uniformly on compacts in $M$ by holomorphic Legendrian embeddings $M\hra \CP^3$. 
This gives the following corollary to Theorem \ref{th:Mergelyan}.

%
%
\begin{corollary} [Mergelyan theorem for Legendrian immersions into $\CP^3$] \label{cor:globalCP3}
Let $S$ be an admissible set in a Riemann surface $M$. Every Legendrian immersion $f:S\to\CP^3$ 
of class $\Ascr^{r+2}(S,\CP^3)$ $(r\ge 2)$ can be approximated in $\Cscr^r(S,\CP^3)$ by holomorphic 
Legendrian immersions $F:M\to\CP^3$, and by holomorphic Legendrian embeddings 
$F:M\hra\CP^3$ if $M$ is an open Riemann surface and $S$ has no holes in $M$. 
Furthermore, $F$ can be chosen to agree with $f$ on any given finite set of points in $S$.
\end{corollary}

It was proved by R.\ Bryant in 1982 (see \cite[Theorem G]{Bryant1982JDG}) 
that every compact Riemann surface embeds into $\CP^3$ as a holomorphic Legendrian curve.
As pointed out in \cite{AlarconForstnericLarusson2019X}, the Runge approximation
theorem also holds for Legendrian curves and immersions in higher dimensional projective spaces 
of odd dimensions, thereby giving the corresponding generalisation of Corollary \ref{cor:globalCP3}.

%
%
%
Another positive result regarding Runge approximation of holomorphic Legendrian
immersions was given by Forstneri\v c and L\'arusson in 
\cite[Corollaries 14 and 16]{ForstnericLarusson2018X}; it pertains
to the projectivised cotangent bundle $X=\P(T^*Z)$ with its standard contact structure
of an arbitrary Oka manifold $Z$ of dimension $\ge 2$. The following is an immediate corollary.

\begin{corollary}\label{cor:PTZ}
Let $S$ be an admissible set without holes in an open Riemann surface $M$.
If $Z$ is an Oka manifold of dimension at least $2$, then every Legendrian 
immersion $S\to X=\P(T^*Z)$ of class $\Ascr^{r+2}(S,X)$ $(r\ge 2)$ can be approximated
in $\Cscr^r(S,X)$ by holomorphic Legendrian embeddings $M\hra X$.
\end{corollary}

With the help of Theorem \ref{th:Mergelyan} we show that approximation by global 
holomorphic Legendrian immersions $M\to X$ is always possible if we allow deformations 
of the complex structure on $M$.  The following result is proved in Section \ref{sec:soft}.

%
%
\begin{theorem}[The soft Oka principle for Legendrian immersions] \label{th:soft}
Assume that $(M,J_0)$ is an open Riemann surface and $K$ is a compact $\Ocal(M)$-convex subset of 
$M$. Let $X$ be a connected complex contact manifold endowed with a distance function
$\dist$ compatible with its manifold topology. Given a continuous map $f:M\to X$ 
which is a $J_0$-holomorphic Legendrian immersion $U\to X$ on an open neighbourhood $U$ of $K$,
there exist for every $\epsilon>0$ a complex structure $J$ on $M$ which agrees with $J_0$ on a 
smaller neighbourhood $U_1\subset U$ of $K$ and a $J$-holomorphic Legendrian embedding 
$F: (M,J)\hra X$ such that 
\[
	\sup_{p\in K} \dist\left(F(p),f(p)\right) < \epsilon
\]
and $F$ is homotopic to $f$ by a homotopy $f_t: M\to X$ $(t\in [0,1])$ such that  
$f_t$ is a holomorphic Legendrian immersion on $U_1$ for every $t\in [0,1]$.
\end{theorem}

Combining Theorems \ref{th:Mergelyan} and \ref{th:soft} we obtain the following corollary.

%
%
\begin{corollary}\label{cor:admissible}
Let $(M,J_0)$ be an open Riemann surface, $S\subset M$ be an admissible compact $\Oscr(M)$-convex 
subset of $M$, $X$ be a complex contact manifold, and $f:S\to X$ be a Legendrian immersion 
of class $\Ascr^{r+2}(S,X)$ for some $r\ge 2$. 
Then there exist a complex structure $J$ on $M$ which agrees with $J_0$ on a 
neighbourhood of $S$ and a $J$-holomorphic Legendrian embedding $F: (M,J)\hra X$ 
which approximates $f$ as closely as desired in the $\Cscr^r(S,X)$ topology.
\end{corollary}

The first systematic treatment of the existence and approximation problem for 
holomorphic maps $M\to X$ from a Stein manifold $M$ (in particular, from an open Riemann surface)
to an arbitrary complex manifold, allowing homotopic deformations of the Stein structure
on the source manifold $M$, goes back to the 2007 papers
by Forstneri\v c and Slapar \cite{ForstnericSlapar2007MZ,ForstnericSlapar2007MRL}.
Results of this type are commonly referred to as instances of the {\em soft Oka principle}. 
Further examples can be found in the papers by 
Alarc\'on and L\'opez \cite{AlarconLopez2013JGEA}, Ritter \cite{Ritter2018JRAM}, and in the surveys 
\cite[Theorem 8.43 and Remark 8.44]{CieliebakEliashberg2012}, \cite[Sects. 10.9--10.11]{Forstneric2017E}, 
and \cite{Forstneric2018Korea}.

%
%
In another direction, we obtain the following corollary by combining Theorem \ref{th:Mergelyan}
with \cite[Theorem 1.3]{AlarconForstneric2019IMRN} by Alarc\'on and Forstneri\v c. 
The assumption in the latter result is that a given Legendrian immersion $f:M\to X$ is holomorphic 
on a neighbourhood of $M$ in an ambient open Riemann surface; this is now guaranteed 
by Theorem \ref{th:Mergelyan}. 
 
%
%
\begin{corollary}[The Calabi-Yau property of holomorphic Legendrian immersions] 
\label{cor:CY}
Assume that $M$ is a compact bordered Riemann surface and $(X,\xi)$ is a complex contact
manifold. Then, every Legendrian immersion $f:M\to X$ of class $\Ascr^4(M,X)$ can be approximated 
uniformly on $M$ by topological embeddings $F:M\hra X$ such that $F|_{\mathring M}:\mathring M \to X$ 
is a complete holomorphic Legendrian embedding. 
\end{corollary}

Here, a map $F:\mathring M \to X$ is said to be complete if the pull-back $F^*g$ of a Riemannian
metric $g$ on $X$ is a complete metric on $\mathring M$. Since the images of the maps $F$ in the above
corollary lie in a compact neighbourhood of $f(M)$, the choice of $g$ is unimportant.

Using the terminology of \cite[Definition 6.1]{AlarconForstnericLarusson2019X},
Corollary \ref{cor:CY} asserts that holomorphic Legendrian curves from compact bordered Riemann surfaces 
to an arbitrary complex contact manifold enjoy the {\em Calabi-Yau property}. This terminology
derives from the classical Calabi-Yau problem for immersed minimal surfaces in Euclidean spaces
$\R^n$, $n\ge 3$; see the papers \cite{AlarconForstneric2019JAMS,AlarconForstneric2019RMI}
and the references therein.

%
%
From the Calabi-Yau property of Legendrian curves in $\CP^3$ we can infer
the analogous property of {\em superminimal surfaces} in the $4$-sphere $S^4$ with the spherical metric.
Among all minimal surfaces in $ S^4$, superminimal surfaces form a natural and important subclass. 
This term was introduced by Bryant \cite{Bryant1982JDG} in 1982, 
although they had been studied much earlier; see the survey in \cite{Forstneric2020JGEA}. 
Superminimal surfaces are characterised geometrically by the fact that the curvature ellipse in 
the normal plane to the surface at each of its points, which is determined by its second fundamental 
form in $S^4$, is a circle centred at the origin (see Friedrich \cite{Friedrich1984,Friedrich1997}).  

%
%
The connection between holomorphic Legendrian curves in $\CP^3$ and superminimal surfaces 
in $S^4$ was discovered by Bryant \cite[Sect.\ 2]{Bryant1982JDG}; here is a brief description.
Identifying $\C^4$ with the quaternionic $2$-plane $\H^2$, the natural projection 
\[
	h:\H^2_*=\H^2\setminus \{0\}\to \H\P^1=S^4
\] 
onto the $1$-dimensional quaternionic projective space 
splits as $h=\pi\circ \pi'$,  where $\pi':\H^2_*=\C^4_*\to \CP^3$ is the standard projection onto 
the complex projective $3$-space and $\pi:\CP^3\to S^4$ is the {\em Penrose's twistor map}
(see Penrose \cite{Penrose1967,Penrose2018}). This is a real analytic  
fibre bundle projection, whose fibres are projective lines in $\CP^3$, such that the differential of $\pi$ induces 
an isometry $d\pi:\xi\to TS^4$ from the contact subbundle $\xi\subset \CP^3$ in the Fubini-Study metric 
onto the tangent bundle of $S^4$ in the spherical metric. (There are actually two twistor bundles
$Z^\pm\to X$ over any Riemannian four-manifold $(X,g)$, but in the case when $X$ is the
sphere $S^4$ with the spherical metric both $Z^\pm$ can be identified with $\CP^3$;
see \cite{Forstneric2020JGEA} and the references therein.) 
Bryant showed that for any Riemann surface, $M$, postcomposition by $\pi$ 
yields a homeomorphism 
\begin{equation}\label{eq:L-SM}
	\pi_*: \Lscr(M,\CP^3) \lra \S(M,S^4)
\end{equation}
from the space $\Lscr(M,\CP^3)$ of holomorphic Legendrian  immersions $M\to\CP^3$ 
onto the space $\S(M,S^4)$ of superminimal immersions $M\to S^4$ of appropriate spin 
\cite[Theorems B, B', D]{Bryant1982JDG}. (The {\em Bryant correspondence} was  
extended by Friedrich \cite{Friedrich1984} to twistor bundles over an arbitrary oriented Riemannian 
four-manifold; see also \cite[Theorem 4.6]{Forstneric2020JGEA}.) 
Since $d\pi|_\xi:\xi\to TS^4$ is an isometry, the projection of a complete Legendrian curve in $\CP^3$ 
is a complete superminimal surface in $S^4$, and vice versa.

The following results are improvements of 
\cite[Corollary 7.3 and Theorem 7.5]{AlarconForstnericLarusson2019X}, respectively.
The difference is that we are now considering superminimal immersions defined on a 
compact  domain, and not on an open neighbourhood of it. Both results use the Bryant
correspondence  \eqref{eq:L-SM}  along with Theorem \ref{th:Mergelyan}.
Note however that the correspondence \eqref{eq:L-SM} essentially depends on  
complete nonintegrability of the contact subbundle $\xi\subset T\CP^3$, and it fails 
on admissible domains containing arcs. In fact, every immersed arc in $S^4$ lifts
isometrically to an immersed $\xi$-Legendrian arc in $\CP^3$.

%
%
\begin{corollary}[Mergelyan approximation theorem for superminimal surfaces in $ S^4$] 
\label{cor:MergelyanS4}
Let $K$ be a compact domain with piecewise $\Cscr^1$ boundary in a Riemann surface $M$.
Every conformal superminimal immersion $X:K\to S^4$ of class $\Cscr^5(K,S^4)$
can be approximated in $\Cscr^2(K,S^4)$ by complete superminimal immersions $Y:M\to S^4$. 
We may choose $Y$ to agree with $X$ to a given finite order at finitely many given points 
in $\mathring K$.
\end{corollary}

%
%
\begin{corollary}[Calabi-Yau theorem for conformal superminimal surfaces in $S^4$]\label{cor:CYS4}
If $M$ is a compact bordered Riemann surface and $X:M\to S^4$ is a conformal superminimal immersion 
of class $\Cscr^5(M,S^4)$, then $X$ can be approximated as closely as desired 
uniformly on $M$ by a continuous map $Y:M\to S^4$ whose restriction to the interior $\mathring M$ 
is a complete, generically injective conformal superminimal immersion, and whose restriction to the boundary 
$bM$ is a topological embedding. 
\end{corollary}

The corresponding results for conformal minimal surfaces in flat Euclidean spaces $\R^n$, $n\ge 3$, 
were obtained in \cite{AlarconDrinovecForstnericLopez2015PLMS,AlarconForstneric2019RMI}.
After the completion of this paper, the author obtained an analogue of Corollary \ref{cor:CYS4}
for an arbitrary self-dual or anti-self-dual Einstein four-manifold $(X,g)$ in place of $S^4$; 
see \cite[Theorem 1.2]{Forstneric2020JGEA}. 

%
%
\section{Tubular neighbourhoods of immersions from admissible sets}\label{sec:tubular}

In this section we prepare some necessary material concerning immersions 
from admissible sets in Riemann surfaces into arbitrary complex manifolds. The main
result that we shall need in the sequel is Lemma \ref{lem:normalbundle}.
We begin by recalling the following result from \cite[Proposition 3.3.2]{Forstneric2017E}.
The hypothesis that the manifold $Z$ be Stein is accidentally missing in the cited source. 
We include a sketch of proof for the convenience of the reader.

%
%
\begin{proposition}
\label{prop:fibre-linearization}
Assume that $\pi: Z\to M$ is a holomorphic submersion from a Stein manifold $Z$ to a complex 
manifold $M$. Denote by $E=\ker d\pi \to Z$ the vertical tangent bundle of $\pi$,
and let $0_z$ denotes the origin in the fibre $E_z$ of $E$ over $z\in Z$.
Then there are an open Stein neighbourhood $\Escr\subset E$ of the zero section of $E$ 
and a holomorphic map $\phi: \Escr \to Z$ such that for every $z\in Z$ we have
$\phi(0_z)=z$ and $\phi$ maps the fibre $\Escr_z=\Escr\cap E_z$ biholomorphically 
onto a neighbourhood of $z$ in the fibre $Z_{\pi(z)}=\pi^{-1}(\pi(z))$.
We may choose $\Escr$ to be Runge in $E$ and to have convex fibres $\Escr_z$, $z\in Z$.
\end{proposition}

\begin{proof}
By an extension of Cartan's Theorem A (see Forster \cite[Corollary 4.4]{Forster1967MZ} 
or Kripke \cite{Kripke1969}) there are finitely many holomorphic vector fields $V_1,\ldots,V_N$
on $Z$ which are tangent to the fibres of $\pi$ and span the vertical tangent space 
$E_z=\ker d\pi_z$ of $Z$  at every point $z\in Z$.  Let $\phi^i_t$ denote the holomorphic 
flow of $V_i$ for the complex time $t$. 
There is an open neighbourhood $\Omega$ of the zero section $Z\times \{0\}^N$
in the trivial bundle $Z\times \C^N$ such that the holomorphic map $\Phi:\Omega\to Z$ given by 
\[
	\Phi(z,t)=\phi^1_{t_1}\circ \cdots \circ \phi^N_{t_N}(z),\quad\  z\in Z,\ 
	t=(t_1,\ldots,t_N)
\]
is well defined on $\Omega$. Clearly, $\Phi(z,0)=z$ and $\pi\circ \Phi(z,t)=\pi(z)$
for all $(z,t)\in\Omega$. Since the vectors $V_i(z)$ span $E_z$ at every point $z\in Z$, the map
\[
	\Theta(z) = \frac{\di}{\di t}\Phi(z,t)\Big|_{t=0} : \C^N\lra E_z
\] 
is surjective, and hence $\ker\Theta$ is a holomorphic vector subbundle 
of $Z\times \C^N$. By Cartan's Theorem B we have 
$Z\times \C^N=\ker\Theta  \oplus E'$ for some holomorphic vector subbundle
$E'\subset Z\times \C^N$. Clearly, the restriction $\Theta|_{E'}:E'\to E$ is a holomorphic 
vector bundle isomorphism, so we may identify $E$ with the subbundle $E'\subset Z\times \C^N$. 
Let $\Escr=\Omega\cap E$. By shrinking $\Omega$ around the zero section if necessary, 
the implicit function theorem shows that the holomorphic map $\phi=\Phi|_{\Escr}:\Escr\to Z$ 
is fibrewise biholomorphic.
\end{proof}

The following lemma provides a coordinate Stein neighbourhood of the graph of a
map of class $\Ascr(S)$ over an admissible set $S$ in a Riemann surface.

%
%
\begin{lemma}\label{lem:Steinnbd}
Assume that $S$ is an admissible set in a Riemann surface $M$, $X$ is a complex
manifold, and $f: S\to X$ is a map of class $\Ascr(S,X)$. Then, the graph
\[
	G_f=\{(p,f(p)): p\in S\}\subset M\times X
\]
of $f$ has a Stein neighbourhood $\wt \Omega \subset M\times X$ which is fibrewise biholomorphic 
to a Stein domain $\Omega \subset M\times \C^n$, $n=\dim X$. More precisely,
there is a biholomorphic map $\Phi: \Omega\to \wt\Omega$
which commutes with the base projections $M\times \C^n\to M$ and $M\times X\to M$.
\end{lemma}

\begin{proof}
By Poletsky's theorem \cite{Poletsky2013} the graph $G_f$ has an open Stein neighbourhood
$Z$ in $M\times X$. (See also \cite[Theorem 32]{FornaessForstnericWold2018}
and the related discussion.) Let $\pi:Z\to M$ denote the restriction of the projection 
$\pi:M\times X\to M$, and let the domain $\Escr \subset E=\ker(d\pi)$ and the map
$\phi: \Escr \to Z$ be as in Proposition \ref{prop:fibre-linearization}.
By \cite[Corollary 9]{FornaessForstnericWold2018} we can approximate $f$ as closely as desired 
uniformly on $S$ by a holomorphic map $\tilde f: U\to X$ defined on an open neighbourhood 
$U\subset M$ of $S$. Denote by $\wt G \subset M\times X$ the graph of $\tilde f$ on $U$. If the approximation 
is close enough and shrinking $U$ around $S$ if necessary, we have that $\wt G\subset Z$. 
Consider the restricted bundle $\wt E=E|_{\wt G}\to \wt G$. Since 
$\wt G$ is biholomorphic to the open Riemann surface $U\subset M$,
the bundle $\wt E$ is holomorphically trivial by the Oka-Grauert principle 
(see \cite[Theorem 5.3.1]{Forstneric2017E}), so we can identify it with $U\times \C^n$, $n=\dim X$.
If the approximation of $f$ by $\tilde f$ is close enough, then the restriction of the map 
$\phi: \Escr \to Z$ to the domain $\Omega:=\Escr \cap \wt E \subset U\times \C^n$ 
provides a biholomorphic map
\begin{equation}\label{eq:Phi}
	\Phi: \Omega \stackrel{\cong}{\longrightarrow} \Phi(\Omega)=
	\wt\Omega\subset Z \subset M\times X
\end{equation}
satisfying the lemma. Indeed, $\Omega$ is a Stein domain in $\wt E \cong U\times \C^n$, 
the map $\Phi$ is fibrewise biholomorphic, and hence biholomorphic onto 
$\wt \Omega = \Phi(\Omega)\subset Z$, and $\wt \Omega$ contains the graph 
$G_f$ of $f$ provided that $\tilde f$ was chosen sufficiently close to $f$ on $S$.
\end{proof}

We denote by $\B^n$ the unit ball of $\C^n$.
%
%
\begin{lemma}\label{lem:normalbundle}
Assume that $S$ is an admissible set in a Riemann surface $M$ and $X$ is a complex
manifold of dimension $n$. Every immersion $f:S\to X$ of class $\Ascr^r(S,X)$ $(r\ge 1)$  
extends to an immersion $F:S\times \B^{n-1}\to X$ of class $\Ascr^r(S\times \B^{n-1}, X)$.
(According to our convention, this means that $F$ is holomorphic on $\mathring S\times \B^{n-1}$
and $F(x,\cdotp): \B^{n-1}\to X$ is holomorphic for every $x\in S$.)
\end{lemma}

\begin{proof}
We may assume that $M$ is an open Riemann surface.
Let $pr_X:M\times X\to X$ denote the projection onto the second component.
Choose a nowhere vanishing holomorphic vector field $V$ on $M$; such exists
by the Oka-Grauert principle (see \cite[Theorem 5.3.1]{Forstneric2017E})
and it defines a trivialisation of the tangent bundle $TM\cong M\times \C$.
For each $x\in S$ the vector $df_x(V_x)\in T_{f(x)}X$ is nonvanishing since 
$f:S\to X$ is an immersion.

Let $\Phi:\Omega\to\wt\Omega\subset M\times X$ be a fibre preserving 
biholomorphic map \eqref{eq:Phi} furnished by Lemma \ref{lem:Steinnbd}, 
where $\Omega$ is a Stein domain in $M\times \C^n$. In particular, the Stein domain $\wt\Omega$
contains the graph of the immersion $f:S\to X$. Hence, there is a unique map $g:S\to \C^n$ 
of class $\Ascr^r(S,\C^n)$ whose graph lies in $\Omega$ such that 
\[
	pr_X\circ \Phi(x,g(x))=f(x)\in X, \quad\   x\in S.
\]
Choose $\rho>0$ such that for every $x\in S$ and $z \in \rho \B^n \subset \C^n$
the point $(x,g(x)+z)$ belongs to $\Omega$. Consider the map $\wt F:S\times \rho\B^n \to X$
given by 
\begin{equation}\label{eq:wtF}
	\wt F(x,z)=pr_X\circ \Phi(x,g(x)+z),\quad\  x\in S,\ z\in \rho\B^n.
\end{equation}
Clearly, $\wt F\in \Ascr^r(S\times \rho\B^n,X)$, and for all $x\in S$ we have that $\wt F(x,0)=f(x)$ 
and $\wt F(x,\cdotp):\rho\B^n \to X$ maps the ball $\rho\B^n$ biholomorphically onto 
a neighbourhood of $f(x)$ in $X$. Thus, $\wt F$ has the required properties, except that 
it is not an immersion. In order to achieve this last condition, we shall restrict it to a suitable
vector subbundle of $S\times \C^n$ of rank $n-1$ and of class $\Ascr^r(S)$.
For each $x\in S$ let $W_x\in \C^n$ be the unique vector satisfying
\[	
	\di_z \wt F_{(x,0)} (W_x) =  df_x(V_x)\in T_{f(x)}X,
\]
where $\di_z$ denotes the partial differential with respect to $z\in \C^n$. 
The map $W:S\to\C^{n}_*$ is of class $\Ascr^{r-1}(S)$. By the Oka-Grauert principle and the 
approximation theorem for vector bundles of class $\Ascr(S)$
(see Heunemann \cite{Heunemann1986I,Heunemann1986II}) we can split the
trivial vector bundle $S\times \C^n$ into a direct sum 
\[
	S\times \C^n = \span_\C(W) \oplus \nu,
\]
where $\span_\C(W)$ is the complex line subbundle determined by $W$ and $\nu$ 
is a complementary trivial holomorphic vector subbundle of rank $n-1$  on a neighbourhood of $S$ in $M$. 
(In the case at hand, this result also follows from Lemma \ref{lem:GL}. A simple proof of Heunemann's approximation 
theorem for complex vector subbundles of class $\Ascr(S)$ in a trivial bundle $S\times \C^n$
by holomorphic vector subbundles over a neighbourhood of $S$ 
(see \cite[Theorem 1]{Heunemann1986I}) can be found in 
\cite[Theorem A.1, pp.\ 248--249]{DrinovecForstneric2007DMJ}.) 
Consider now the map
\[
	F=\wt F|_{\nu\cap (S\times \rho\B^n)}: \nu\cap (S\times \rho\B^n)  \to X
\]
of class $\Ascr^r$. For each $x\in S$ we have $F(x,0)=f(x)$ and the differential of $F$ in the vertical 
direction maps the fibre $\nu_x$ onto a hyperplane in $T_{f(x)}X$ complementary to the vector 
$df_x(V_x)$. Decreasing $\rho>0$ if necessary, it follows that $F$ 
is an immersion of class $\Ascr^r$ with trivial normal bundle.
After a change of coordinates mapping $\nu\cap (S\times \rho\B^n)$ onto $S\times \B^{n-1}$,
we get an immersion $F:S\times \B^{n-1}\to X$ satisfying the lemma.
\end{proof}

%
%

\section{Preliminaries} 
\label{sec:ODE}
In this section we prepare the necessary background for the proof of Theorem \ref{th:Mergelyan}.
In Subsect.\ \ref{ss:ODE} we recall some basic facts concerning 
solutions of ordinary holomorphic differential equations, with emphasis on the
case when the domain is an admissible set in a Riemann surface.
In Subsect.\ \ref{ss:period} we recall the notion of the period map,  which
plays an important role in the deformation theory of solutions of holomorphic differential equations.
In Subsect.\ \ref{ss:Runge} we show that every admissible set $S$ in a Riemann surface 
admits a homology basis consisting of finitely many closed curves whose union is Runge in $S$.
In Subsect.\ \ref{ss:degree} we recall a basic result on the topological
degree of a map. Finally, in Subsect.\ \ref{ss:generators} we recall a result of Arens concerning
generators of the algebra $\Ascr^r(S)$.

%
%
\subsection{Holomorphic differential equations on admissible sets}\label{ss:ODE}

Assume that $M$ is an open Riemann surface. Fix a holomorphic immersion $z:M\to \C$ 
furnished by Gunning and Narasimhan \cite{GunningNarasimhan1967}. Such an immersion 
provides a local holomorphic coordinate on a neighbourhood of any given point of $M$.
Let $S=K\cup E$ be an admissible set in $M$ (see Def.\ \ref{def:admissible}).
We shall need some basic results on the existence and behaviour of solutions of
ordinary differential equations 
\begin{equation}\label{eq:ODE}
	 dw = V(p,w,t) dz,  \quad\  w(p_0,t)=w_0, 
\end{equation}
where the independent variable is $p\in S$, the dependent variable $w$ belongs to some disc
$\triangle \subset \C$ around the origin, the differentials $dz$ and $dw$ are taken with respect
to $p$, $t=(t_1,\ldots,t_l)$ is a complex parameter in a ball $B\subset \C^l$ around the origin, 
and $V$ is a function of class $\Cscr^r$ on $S\times \triangle \times B \subset M\times \C^{l+1}$ 
for some $r\ge 1$ which is holomorphic on the interior $\mathring S \times \triangle \times B$. 
The function $V$ may be thought of as a nonautonomous vector field $V\!\frac{\di}{\di w}$ of 
type $(1,0)$ on the $w$-space, of class $\Cscr^r$ in $(p,w)\in S\times \triangle$ and
holomorphic over $\mathring S$, depending holomorphically on the parameter $t\in B$.
On a neighbourhood of a point $p_0\in S$, using the holomorphic immersion $z:M\to\C$
as a local coordinate near $p_0$ and setting $z_0=z(p_0)\in\C$, the equation
\eqref{eq:ODE} assumes a more familiar form
\begin{equation}\label{eq:ODElocal}
	\frac{dw}{dz} = V(z,w,t),  \quad\  w(z_0,t)=w_0.
\end{equation}
(The function $V$ in \eqref{eq:ODElocal} is obtained from the one in \eqref{eq:ODE} by locally expressing 
$p=p(z)$.) This is an ordinary differential equation for $w=w(z,t)$ as a function of $z$, with the parameter $t$.
The precise local nature of this equation depends on the location of the point $p_0\in S$
where we are considering it. If $p_0$ is an interior point of $S$, then \eqref{eq:ODElocal} is a holomorphic 
differential equation near $p_0$ in the local coordinate $z$ centred at $z_0=z(p_0)\in\C$.
Such an equation admits a local holomorphic solution which is uniquely
determined by an initial condition $w(z_0,t)=w_0$ and depends holomorphically on 
$(z_0,w_0,t)$ (see E.\ Hille \cite[Chapter 2]{Hille1976}). 
One may find a local solution in terms of the power series expansion 
\[
	w(z,t) = w_0+\sum_{k=1}^\infty c_k (z-z_0)^k. 
\] 
The coefficients $c_k=c_k(z_0,w_0,t)$ are uniquely determined by the equation \eqref{eq:ODElocal}. 
By the domination method of A.\ L.\ Cauchy or E.\ Lindel\"of one can show that the power series 
converges in a disc around the point $z_0$, and it is possible to estimate its radius in terms of $V$. 
(See E.\ Hille \cite[Sect.\ 2.6]{Hille1976} or the book by E.\ Lindel\"of \cite{Lindelof1905} from 1905.) 

This method does not apply at boundary points of the domain $K$ in the admissible
set $S$, or over the arcs $ E\subset S$. Let us now explain an alternative approach 
which works up to the boundary of $K$. (We shall consider the equation on the arcs
in $S\setminus K$ later.) Write the variables and the vector field in the form
\[
	z=x+\imath y,\quad w=w_1+\imath w_2, \quad V=V_1+\imath V_2
\]
with real components. 
The ordinary complex differential equation \eqref{eq:ODElocal} is equivalent to the following system 
of two real partial differential equations for $w=(w_1,w_2)$:
\begin{eqnarray}\label{eq:system1}
	\frac{\di w_1}{\di x} &=& V_1,       \quad\ \ \ \frac{\di w_2}{\di x} = V_2, \\
	\frac{\di w_1}{\di y} &=& -V_2,      \quad \frac{\di w_2}{\di y} = V_1.  \label{eq:system2}
\end{eqnarray}
A calculation shows that the vector fields 
\[
	X=\frac{\di}{\di x} + V_1\frac{\di}{\di w_1}+V_2\frac{\di}{\di w_2}, \qquad 
        Y=\frac{\di}{\di y}  -V_2\frac{\di}{\di w_1}+V_1\frac{\di}{\di w_2}
\] 
commute when the function $V=V_1+\imath V_2$ is holomorphic in $(z,w)$, 
and by continuity this persists up to the boundary of $K$. 
Hence, the flow $\phi_x$ of $X$ commutes with the flow $\psi_y$ of $Y$ on their domains
of definition. The local solution $w=w(z)$ of the initial value problem \eqref{eq:ODElocal} 
is then the composition of these two flows, projected to the $w$-space:
\begin{equation}\label{eq:flows}
	w(z_0+x+\imath y) = pr_w\circ \phi_x \circ \psi_y (z_0,w_0) 
	= pr_w\circ \psi_y \circ \phi_x  (z_0,w_0).
\end{equation}
Indeed, differentiation of \eqref{eq:flows} on $x$ and $y$ gives the equations \eqref{eq:system1}, \eqref{eq:system2} 
which are equivalent to \eqref{eq:ODElocal}. Clearly, a solution \eqref{eq:flows} also exists at the 
boundary points of $K$ provided $bK$ is piecewise $\Cscr^1$. This gives local, and often also global
holomorphic solutions of \eqref{eq:ODE} in terms of flows of vector fields, an ostensibly simpler problem. 
The same method applies if the vector field $V$ depend holomorphically on a parameter $t$.

Finally, we can parameterise an arc or a closed curve $\Gamma$ in $E=\overline{S\setminus K}$ 
by an injective immersion $p=h(s)$, $s\in [0,1]$, except that $h(0)=h(1)$ if $\Gamma$ is a closed 
Jordan curve. The differential equation \eqref{eq:ODE} then takes the following form on $\Gamma$:
\begin{equation}\label{eq:Jordancurve}
	\frac{d}{ds}  w(h(s),t) = V(h(s),w(h(s),t),t) \frac{d}{ds}  z(h(s)),\quad\  w(h(0),t)=w_0. 
\end{equation}
This is an ordinary differential equation for the function $s\mapsto w(h(s),t)$ on $s\in [0,1]$.

It is classical that solutions of the initial value problem \eqref{eq:ODElocal} depend holomorphically
on the initial point $(z_0,w_0)$ in the open set where $V$ is holomorphic (see \cite[Theorem 2.8.2]{Hille1976}). 
If $V$ is of class $\Ascr^{r}(K\times \Delta)$  then the flows  in \eqref{eq:flows} are of class 
$\Cscr^{r}$ up to the boundary of $K$, and hence so are the solutions. 
Furthermore, the solutions on $D=\{|z-z_0| <\rho\}\cap K$ 
corresponding to a pair of initial values $w_0,w_1$ at $z_0$ satisfy an estimate
\begin{equation}\label{eq:estimate1}
	|w(z;z_0,w_0)-w(z;z_0,w_1)| \le c_0 |w_0-w_1| e^{c|z-z_0|} 
\end{equation}
as long as their graphs remain in $D\times \triangle$. Here, $c>0$ is the Lipshitz constant 
for $V$ with respect to the variable $w$:
\[
	|V(z,w)-V(z,w')| \le c |w-w'|,\quad\  z\in D,\ w,w'\in \triangle,
\]
and the constant $c_0\ge 1$ reflects the geometry of $D$ (we may take $c_0=1$ if $D$ is convex).
This follows from Gr\"onwall's inequality; 
see \cite[Theorem 2.8.1]{Hille1976} and \cite[Lemma 1.9.3]{Forstneric2017E}.
Covering $S$ with finitely many discs such that the immersion $z:M\to\C$ gives
a local holomorphic coordinate of each of them, we get a similar estimate \eqref{eq:estimate1} globally on $S$.

Gr\"onwall's inequality can also be used to estimate the difference between solutions 
of a perturbed equation and those of the original equation. 
Explicitly, if a function $\wt V$ of class $\Ascr^{r}(S\times \triangle\times B)$
is uniformly close to $V$ on $S\times \triangle\times B$ then for any $p_0\in S$ the solution 
$w(p;p_0,w_0,t)$ of \eqref{eq:ODE} is uniformly close to the solution $\wt w(p;p_0,w_0,t)$ 
of the same equation for $\wt V$, provided that both solutions exist and their graphs remain in the given domain.
We refer to \cite[Lemma 1.9.4]{Forstneric2017E} for a precise estimate in a similar context.
From the equation \eqref{eq:ODE} we then infer that the solutions $w(p;p_0,w_0,t)$ and $\wt w(p;p_0,w_0,t)$
are also $\Cscr^1$ close to each other.
More generally, by a prolongation of the system we infer that if $\wt V$ is $\Cscr^r$ close to 
$V$ then their solutions for the same initial values are $\Cscr^{r+1}$ close to each other.

%
%
\subsection{The period map}\label{ss:period}

So far we have been considering local solutions of the equation \eqref{eq:ODE}.
If $S$ is a compact simply connected domain with piecewise $\Cscr^1$ boundary 
in a Riemann surface $M$, then by the uniqueness of local solutions they 
amalgamate into a global solution on $S$ provided they remain in the domain of the vector field $V$.
In particular, given a global solution $w(p;p_0,w_0)$ of \eqref{eq:ODE} on $S$
(neglecting the parameter $t$ for the moment), any number $w_1\in\C$ sufficiently
close to $w_0$ determines a solution $w(p;p_0,w_1)$ of the same equation on all of $S$
since, by \eqref{eq:estimate1}, the solution $w(p;p_0,w_1)$ remains close to the original one.
 
The situation is rather different on an admissible set $S$ with nontrivial 
fundamental group; equivalently, with nontrivial first homology group $H_1(S,\Z)$. 
(Note that $H_1(S,\Z)=\Z^l$ is a free abelian group on finitely many generators.) 
In this case, an important role in the global existence and perturbation theory of solutions 
is played by the period map along homologically nontrivial closed 
curves in $S$. We now recall this notion.

Assume that $C$ is a closed piecewise smooth Jordan curve in $S$, and choose a parameterisation
$h:[0,1]\to C$ with $h(0)=h(1)=p_0\in C$. In the parameter $s\in [0,1]$, the equation
\eqref{eq:ODE} with the initial condition $w(p_0,t)=w_0$ takes the form \eqref{eq:Jordancurve}.
Assume that the solution $w(s;w_0,t)$ with $w(0;w_0,t)=w_0$ exists for all $s\in [0,1]$. The number 
\begin{equation}\label{eq:period}
	\Pscr_C(p_0,w_0,t) = w(1;w_0,t)-w(0;w_0,t)= w(1;w_0,t)-w_0
\end{equation}
is called the {\em period} along $C$ for the data $(p_0,w_0,t)$. 
This period is independent of the choice of oriented parameterisation of $C$. 
A necessary condition for the existence of a single-valued solution of the equation \eqref{eq:ODE}
along the curve $C$ for the data $(p_0,w_0,t)$ is that 
\[
	\Pscr_C(p_0,w_0,t)=0,
\]
which means that the map $s\mapsto w(s;w_0,t)$ is $1$-periodic.
Conversely, if this holds then the equation \eqref{eq:ODE} has a single-valued 
solution on an annulus around $C$ intersected with $S$.
By varying the initial value at the point $p_0$ we obtain the map 
\[
	\zeta\longmapsto \Pscr_C(p_0,\zeta,t)\in\C, \quad\  \zeta\in \C\ \text{near}\ w_0,
\]
called the {\em Poincar\'e first return map} of the closed orbit $s\mapsto w(s;w_0,t)$.
This map describes the dynamics of orbits in a neighbourhood of the given
periodic orbit. The return map vanishes identically if and only if all nearby
solutions are periodic on $C$, and in such case their graphs form a foliation of the phase space
near the graph of the initial solution.

Let us consider more closely the case when $S=K$ is a connected compact domain 
with piecewise $\Cscr^1$ boundary and nontrivial homology group $H_1(K,\Z)$. Then, 
$H_1(K,\Z)=\Z^l$ is a free abelian group whose generators are
represented by smooth closed Jordan curves $C_1,\ldots,C_l\subset \mathring K$.
It is classical (see e.g.\ \cite{FarkasKra1992})
that the generating curves can be chosen to have a common base point $p_0\in \mathring K$,
to satisfy $C_i\cap C_j=\{p_0\}$ for $1\le i\ne j\le l$, and such that their union $C=\bigcup_{i=1}^l C_i$ 
is a Runge set in $K$, i.e., $\mathring K\setminus C$ has no connected components with 
compact closure in $\mathring K$. 
(The Runge condition will be important in our proof since it implies that functions 
$C\to\C$ of class $\Cscr^r(C)$ $(r\in \Z_+)$ can be approximated in $\Cscr^r(C)$ by functions 
holomorphic on a neighbourhood of $K$ in $M$; see \cite[Theorem 16]{FornaessForstnericWold2018}.) 
It follows that $K$ admits a deformation retraction onto $C$, and the equation \eqref{eq:ODE} 
has a single-valued solution on $K$ if and only if the solution remains in the domain 
of the equation and has vanishing period on each of the curves $C_1,\ldots, C_l$
in the homology basis of $K$.

%
%
\subsection{Runge homology basis of an admissible set} \label{ss:Runge}

When considering a holomorphic differential equation on a general admissible set $S=K\cup E$ in 
a Riemann surface $M$, we must control the periods of solutions on closed curves 
in a basis for the homology group $H_1(K,\Z)$, and also on the arcs in $S\setminus K$. 
To this end, we now construct a special Runge homology basis of an admissible set. 
Although the existence of such a homology basis was used before (see in particular the papers
\cite{AlarconForstneric2014IM,AlarconForstneric2019JAMS,AlarconForstnericLopez2017CM}), 
a detailed construction has not been given. The case when $S=K$ is a union of domains is classical; 
see Farkas and Kra \cite{FarkasKra1992}.

Given a Riemannian distance function $\dist$ on $M$ and a number $\epsilon>0$, the set
\begin{equation}\label{eq:Sepsilon}
	S_\epsilon = \{p\in M: \dist(p,S)<\epsilon\}
\end{equation}
is an open neighbourhood of $S$ which admits a deformation retraction onto $S$ provided
$\epsilon>0$ is small enough; we shall call such $S_\epsilon$ a {\em regular neighbourhood} of $S$. 

%
%
\begin{lemma}\label{lem:homology-basis-admissible}
A connected admissible set $S$ in a Riemann surface $M$  
has finitely generated first homology group $H_1(S,\Z)\cong \Z^l$,
and there is a homology basis $\Ccal=\{C_1,\ldots,C_l\}$ consisting 
of closed piecewise smooth Jordan curves in $S$ such that the compact set $|\Ccal|=\bigcup_{i=1}^l C_i$ 
is connected and Runge in any regular neighbourhood $S_\epsilon$ \eqref{eq:Sepsilon} of $S$,
and each curve $C_i\in \Ccal$ contains a nontrivial arc $I_i$ disjoint from $\bigcup_{j\ne i} C_j$.
\end{lemma}

\begin{proof}
Let $S=K\cup E$ with $K=\bigcup_{i=1}^m K_i$, where $K_1,\ldots,K_m$ are the connected 
components of $K$. If $K=\varnothing$ then $S$ is a single arc or closed curve and the result is trivial. 
Assume now that $K\ne \varnothing$. Since $S$ is connected, 
$E=\bigcup_{k=1}^{n} E_k$ is a union of finitely many smooth pairwise disjoint arcs $E_k$.
In each component $K_i$ of $K$ we choose an interior point $q_i\in \mathring K_i$, which we shall 
call the {\em vertex} of $K_i$. The boundary $bK_i=\bigcup_{j=1}^{m_i}\Gamma_{i,j}$ consists
of finitely many closed Jordan curves for some $m_i\ge 1$. 

The standard construction (see \cite{FarkasKra1992}) gives for each $i=1,\ldots,m$ 
a basis of $H_1(K_i,\Z)$ consisting of finitely many Jordan curves 
in $\mathring K_i$ passing through the vertex $q_i$ and not intersecting elsewhere
whose union is Runge in $K_i$; we put all these curves in the family $\Ccal$ under construction.

For every $i=1,\ldots,m$ and $j=1,\ldots, m_i$ we choose a pair of distinct points
$a_{i,j},b_{i,j}\in \Gamma_{i,j}$ such that $b_{i,j}\notin E$.
We connect $a_{i,j}$ to $q_i$ by a smooth embedded arc $A_{i,j} \subset \mathring K_i\cup\{a_{i,j}\}$,  
chosen such that these arcs do not intersect each other, nor any of the chosen curves in the 
homology basis for $K_i$, except at the vertex $q_i$. 

Recall that $E_1,\ldots,E_n$ are the connected components (arcs) of $E=\overline {S\setminus K}$.
Every arc $E_k$ in this collection is of one of the following three types. 

{\em Case 1:} $E_k$ is attached to $K$ with only one endpoint. Clearly, such arcs do not affect
the homology. Let $S_0$ denote the admissible set obtained by attaching all such arcs to $K$.

{\em Case 2:} The endpoints of $E_k$ lie in connected components
$\Gamma_{i,j_1}$, $\Gamma_{i,j_2}$ of $bK_i$ for some $i\in\{1,\ldots,m\}$.
(These components may be the same.)
In this case, a new homologically essential closed curve in $S$ is obtained by connecting
the endpoints of $E_k$ inside $K_i$ as follows. Having traversed $E_k$ to its endpoint
in $\Gamma_{i,j_2}$, continue to the point $a_{i,j_2}$ along the unique arc in 
$\Gamma_{i,j_2}$ which does not contain $b_{i,j_2}$, then go from $a_{i,j_2}$ to the vertex 
$q_i$ along the arc $A_{i,j_2}$, continue from $q_i$ to $a_{i,j_1}$ along $A_{i,j_1}$, and 
finally connect $a_{i,j_1}$ to the initial point of $E_k$ by the arc in $\Gamma_{i,j_1}$ 
not containing $b_{i,j_1}$. We add all closed curves obtained in this way to the family $\Ccal$, and we 
denote by $S_1$ the admissible set obtained by adding all such arcs $E_k$ to 
the set $S_0$ from the previous step. 
Note that $S_1$ still has the same number of connected components as $K$, namely $m$.

{\em Case 3:} The endpoints of $E_k$ belong to different connected components
of $K$ (and hence of the admissible set $S_1$); let us call such an arc a {\em bridge}. 

Let $S_2$ denote a connected admissible set obtained by attaching to $S_1$ 
a collection of bridges such that removing any one of them disconnects $S_2$. 
Paint these bridges black. (Such $S_2$ need not be unique.) 
Clearly, the inclusion $S_1\hra S_2$ induces an isomorphism 
of the homology groups $H_1(S_1,\Z)\cong H_1(S_2,\Z)$. 

We paint the remaining bridges red. For every red bridge $E_k$ 
there are pairwise distinct bridges $E_k=E_{k_1},\ldots,E_{k_s}$, all but $E_k$ black,
and connected components $K_{i_1},\dots, K_{i_s}$ (islands) of $K$ such that 
$E_{k_1}=E$ connects $K_{i_1}$ to $K_{i_2}$, $E_{k_2}$ connects  
$K_{i_2}$ to $K_{i_3}$, etc., until the cycle closes with the bridge $E_{k_s}$ connecting 
$K_{i_s}$ back to $K_{i_1}$. We obtain a new closed curve in $S$ 
by connecting the end point of each bridge $E_{k_j}$ in the above sequence within the 
domain $K_{i_j}$ to the initial point of the next bridge $E_{k_{j+1}}$, where
$E_{k_{s+1}}=E_1$. The connecting curves in domains $K_{i_j}$ are chosen as in 
Case 2 above. By attaching all red bridges to $S_2$ we obtain the original admissible set $S$, 
and by adding the corresponding closed curves to $\Ccal$ furnishes a homology basis of $S$.
(See Figure \ref{fig:bridges}.) 
By the construction, every red bridge (see Case 3) is contained in precisely one curve in 
the family $\Ccal$. The same is true for each of the arcs in Case 2. Hence, every curve in 
$\Ccal$ contains an arc which is disjoint from all other curve in $\Ccal$.

\begin{figure}[h]
\begin{center} 
	\includegraphics[scale=0.22]{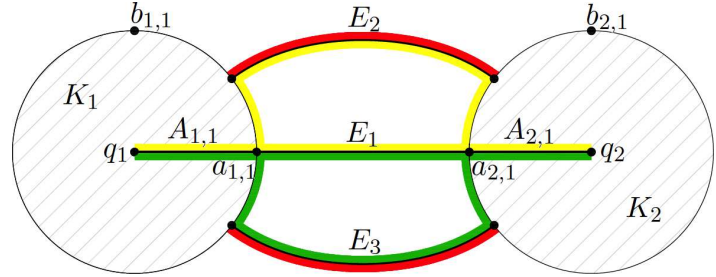} 
\end{center}
\caption{Two islands connected by three bridges}
\label{fig:bridges}
\end{figure}

Let $|\Ccal|$ denote the union of all curves in $\Ccal$.
By the construction, any point $p\in K_i\setminus |\Ccal|$ $(i=1,\ldots, m)$ can be connected by an arc in 
$K_i\setminus |\Ccal|$ to the point $b_{i,j}\in \Gamma_{i,j}\subset bK_i$ for some $j\in \{1,\ldots, m_i\}$. 
Hence, the set $|\Ccal|$ has no holes in $S$, so it is Runge. We can make $|\Ccal|$ connected
by modifying each closed curve $C\in \Ccal$ to pass through the vertex $q_1\in K_1$.
Indeed, every such curve $C$ passes through one of the vertices $q_i\in K_i$, so it suffices
to connect $q_i$ to $q_1$ as described in Case 3 above, using only black bridges when passing 
between different connected components of $K$.
\end{proof}

%
%
\subsection{Topological degree of a map}\label{ss:degree}
We recall the notion of the {\em topological degree} of a continuous map
which was defined and studied by L.\ E.\ J.\  Brouwer in 1911, \cite{Brouwer1911}.
For a modern treatment, see e.g.\ M.\ Hirsch \cite{Hirsch1976} or 
J.\ Milnor \cite{Milnor1997}. 

Let $f:M\to N$ be a smooth map between closed (compact without boundary)
connected oriented manifolds of the same dimension $n$. Then, every regular 
value $q\in N$ of $f$ has finitely many preimages, and its degree $\deg(f)$ is the 
signed number of points in the fibre $f^{-1}(q)$ taking into account the orientations.
(For nonorientable manifolds one can introduce the notion of degree modulo $2$.)
It turns out that this number is independent of $q\in N$. 
Furthermore, a pair of homotopic maps $M\to N$ have the same degree, 
so the degree can also be defined for continuous maps between compact topological manifolds. 

Let $D$ be a compact domain in $\R^n$ which is a topological manifold 
with coherently oriented boundary $bD$. Given a continuous map  $f:D\to\R^n$ and 
a point $q\in \R^n\setminus f(bD)$, we define
$
	\deg(f,q)\in\Z
$
as the topological degree of the map $\pi\circ f:bD\to S^{n-1}$, where 
$\pi:\R^n\setminus \{q\}\to S^{n-1}$ is the retraction $\pi(x)=\frac{x-q}{|x-q|}$ onto the sphere. 
If $f$ is smooth and $q$ is a regular value of $f$, then $\deg(f,q)$ is the signed 
number of points in the fibre $f^{-1}(q)$. This gives the following observation. 

%
%
\begin{proposition}\label{prop:degree}
Assume that $D$ is a compact domain in $\R^n$ which is a topological manifold with boundary $bD$. 
If $f:D\to \R^n$ is a continuous map such that $f(bD)\subset \R^{n}_*=\R^n\setminus \{0\}$ and the map 
$\frac{f}{|f|} :bD\to S^{n-1}$ has nonzero degree, then there is a point $p \in D$ with $f(p)=0$. 
In particular, if $0\in \mathring D$ and $f_t:D\to\R^n$ $(t\in [0,1])$ is a homotopy with $f_0=\Id_D$ 
such that $f_t(bD)\subset \R^n_*$ for all $t\in [0,1]$, then $0\in f_t(D)$ for all $t\in[0,1]$. 
\end{proposition}

%
%
\subsection{Generators of $\Ascr^r(S)$}\label{ss:generators}
We shall need the following lemma. 

\begin{lemma}\label{lem:Arens}
Let $S$ be an admissible set in a Riemann surface (see Definition \ref{def:admissible}). 
Given functions $f_1,\ldots,f_m\in\Ascr^r(S)$ $(r\in\Z_+)$ without common zeros, 
there are functions $g_1,\ldots,g_m\in\Ascr^r(S)$ such that
\begin{equation}\label{eq:fg}
	f_1g_1+f_2g_2+\cdots + f_mg_m=1.
\end{equation}
\end{lemma}

\begin{proof} For $r=0$ this follows from the result of R.\ Arens  \cite{Arens1958}
which states that every maximal ideal of the algebra $\Ascr(S)$ is given by the evaluation
at a point of $S$. (When $S=\overline\D$ is the closed unit disc in $\C$, this is a special case
of results of W.\ Rudin \cite{Rudin1957} who described closed ideals of the disc algebra $\Ascr(\cd)$.)
Hence, a collection of functions in $\Ascr(S)$ without a common zero spans $\Ascr(S)$, 
so \eqref{eq:fg} holds. (Arens's result applies in the more general situation when $S$ is a 
compact set in a Riemann surface $M$, $U$ is an  open set in $M$ contained in $S$, 
and $\Ascr(S,U)$ is the algebra of continuous functions on $S$
which are holomorphic on $U$. Here we are taking $U=\mathring S$.) 

Suppose now that $r>0$ and $f_1,\ldots,f_m\in\Ascr^r(S)$ have no common zeros.
Let $g_1,\ldots,g_m\in \Ascr(S)$ satisfy \eqref{eq:fg}. By Mergelyan's theorem
we can approximate each $g_j$ uniformly on $S$ by a function $\tilde g_j\in\Oscr(S)$. 
If the approximations are close enough then the function $h=\sum_{j=1}^m f_j\tilde g_j\in \Ascr^r(S)$ 
has no zeros on $S$, and the functions $G_j=\tilde g_j/h\in \Ascr^r(S)$ for $j=1,\ldots,m$
satisfy $\sum_{j=1}^m f_j G_j=1$.
\end{proof}

We shall need the following generalisation of Lemma \ref{lem:Arens}, analogous to
\cite[Lemma 2.1]{AlarconForstneric2019IMRN}. The proof given there applies verbatim if we
replace the use of Cartan's Theorems A and B over an open Riemann surface
by the corresponding results of Heunemann \cite{Heunemann1986I,Heunemann1986II}
for complex vector bundles of class $\Ascr(S)$ (see the proof of Lemma \ref{lem:normalbundle}).

\begin{lemma}\label{lem:GL}
Let $S$ be an admissible set in a Riemann surface $M$, and let $A$ be an  
$m\times p$ matrix-valued function on $S$, $1\le m<p$, of class $\Ascr^r(S)$
which has maximal rank $m$ at every point of $S$. Then there exists a map 
$B: S\to GL_p(\C)$ of class $\Ascr^r(S)$ such that $A(p)\cdotp B(p)=(I_m,0)$ 
holds for all $p\in S$, where $I_m$ is the $m\times m$ identity matrix. 
\end{lemma}

%
%

\section{Proof of Theorem \ref{th:Mergelyan}}\label{sec:proof}

Let $(X,\xi)$ be a complex contact manifold with $\dim X=2n+1\ge 3$, and let $\eta$ be a 
holomorphic $1$-form on $X$ with values in the normal line bundle $TX/\xi$ such that $\xi=\ker\eta$;  
see \eqref{eq:eta}. Assume that $S$ is an admissible set in a Riemann surface $M$ and 
$f\in \Ascr^{r+2}(S,X)$ $(r\ge 2)$ is an immersed $\xi$-Legendrian curve.
By Lemma \ref{lem:Steinnbd}, the graph
\[
	G_f=\{(\zeta,f(\zeta)):\zeta \in S\}\subset M\times X
\]
has an open Stein neighbourhood $\wt \Omega \subset M\times X$.
By Lemma \ref{lem:normalbundle}, $f$ extends to an immersion $F:S\times \B^{2n}\to X$ 
of class $\Ascr^{r+2}(S\times \B^{2n},X)$ whose graph is contained in $\wt \Omega$.
By standard results (see e.g.\ \cite[Lemma 4.3]{HormanderWermer1968} or 
\cite[Proposition 5.55]{CieliebakEliashberg2012}), $F$ extends to an immersion $F:U\times \B^{2n}\to X$ 
of class $\Cscr^{r+2}$, where $U\subset M$ is an open neighbourhood of $S$,
which is asymptotically holomorphic to order $r+1$ on $S\times \B^{2n}$, meaning that 
\begin{equation}\label{eq:dibarflat}
	D^{r+1}(\dibar F)=0 \ \ \text{holds on $S\times \B^{2n}$}.
\end{equation}
Here, $D^r$ denotes the total derivative of order $r$. In the case at hand, we can 
be more precise. Choosing a $\dibar$-flat extension of the map $f:S\to X$, we see from 
the proof of Lemma \ref{lem:normalbundle} that there is an $F$ as above such that, in addition
to \eqref{eq:dibarflat}, the map $F(x,\cdotp): \B^{2n}\to X$ is holomorphic for every $x\in U$. 

The holomorphic contact $1$-form $\eta$ on $X$ assumes values in the normal bundle $L=TX/\xi$
of the contact structure, which is a possibly nontrivial holomorphic line bundle. However, in our 
analysis we may assume that $L$ is trivial and $\eta$ is scalar-valued, which is seen as follows. 
By Lemma \ref{lem:Steinnbd} there are a Stein neighbourhood
$\wt \Omega\subset M\times X$ of the graph of the immersion $f:S\to X$ and 
a fibre-preserving biholomorphic map $\Phi:\Omega\to\wt\Omega\subset M\times X$ 
\eqref{eq:Phi} from a Stein domain $\Omega$ in $M\times \C^n$. 
Since $S$ is a compact set in an open Riemann surface, we can choose 
$\wt \Omega$ to have vanishing cohomology group $H^2(\wt\Omega,\Z)=0$.
By Oka's theorem (see \cite[Corollary 5.2.3]{Forstneric2017E}) this implies that every 
holomorphic line bundle on it is holomorphically trivial  \cite[Corollary 5.2.3]{Forstneric2017E}. 
By the proof of  Lemma \ref{lem:normalbundle}, the immersion $F:U\times \rho \B^{2n}\to X$ 
has range in $\wt \Omega$ and the claim follows.

Consider the 1-form
\begin{equation}\label{eq:beta}
	\beta = F^*\eta =\beta^{1,0}+\beta^{0,1}
\end{equation}
of class $\Cscr^{r+1}(U\times \B^{2n})$. Since $\eta$ is a holomorphic $1$-form, 
\eqref{eq:dibarflat} implies that the coefficients of the $(1,0)$-part $\beta^{1,0}$ 
are asymptotically holomorphic to order $r+1$ on $S\times \B^{2n}$, 
while the coefficients of $\beta^{0,1}$ vanish to order $r+1$ on $S\times \B^{2n}$. 
Let us call such a $1$-form {\em asymptotically holomorphic to order $r+1$} along $S\times \B^{2n}$.
For the same reason, since $d\eta$ is a holomorphic $2$-form, the differential 
$d\beta = F^*(d\eta)$ is a $2$-form of class $\Cscr^{r+1}(U\times \B^{2n})$ 
(there is no loss of derivatives) which is asymptotically holomorphic to order $r+1$ along $S\times \B^{2n}$. 
Finally, 
\[
	\beta\wedge(d\beta)^n=F^*(\eta\wedge (d\eta)^n) \ne 0
\] 
is a nowhere vanishing $(2n+1)$-form of class $\Cscr^{r+1}(U\times \B^{2n})$ which 
which is asymptotically holomorphic to order $r+1$ along $S\times \B^{2n}$. 
In this sense, $\beta$ is a complex contact form on $S\times \B^{2n}$. 
(Note that $\beta$ is a holomorphic contact form on  $\mathring S\times  \B^{2n}$.
A more precise treatment of asymptotically holomorphic contact forms can be
found in \cite[Sect.\ 3]{Forstneric2020JSG}; see in particular \cite[Lemma 3.2]{Forstneric2020JSG}.)

%
%
We shall need the following partial normal form of $\beta$ along $S\times \{0\}^{2n}$.
    
%
%
\begin{lemma}\label{lem:NF}
Let $\beta$ be as in \eqref{eq:beta}, and let $z:M\to\C$ be a holomorphic immersion.
There are fibre coordinates $\zeta=(w,y,\zeta')$ with $\zeta'=(\zeta_3,\cdots,\zeta_{2n})$ on $S\times \rho\B^{2n}$
and a nowhere vanishing function $h\in \Ascr^{r-1}(S\times \rho\B^{2n})$ for some $0<\rho<1$ such that 
\begin{equation}\label{eq:beta-normal}
	\frac{1}{h}\beta  = dw -  y dz + \sum_{j,k=3}^{2n} c_{j,k}\zeta_k  \, d\zeta_j  + \wt \beta
	\quad \text{on}\ S\times \rho\B^{2n},
\end{equation}
where $c_{j,k}\in\Ascr^{r-1}(S)$ and $\wt \beta$ is a $1$-form of class $\Ascr^{r-1}$ 
containing terms $y d\zeta_j$ and $\zeta_j dy$ for $j=2,\ldots,2n$, terms $O(|\zeta|) dw$, and $O(|\zeta|^2) dz$. 
The change of coordinates which brings $\beta$ in this form is a fibre-preserving transformation 
of class $\Ascr^{r}$ on a neighbourhood of $S\times \{0\}^{2n}$ in $S\times \C^{2n}$ which is $\C$-linear 
in the fibre variable and keeps $S\times \{0\}^{2n}$ fixed. 
\end{lemma}

%
%
\begin{remark}\label{rem:normalisation}
Over the interior of $S$ where $\beta$ is a holomorphic contact form, it has a full 
Darboux-type holomorphic normalisation along $\mathring S\times \{0\}^{2n}$ of the form
\[
	\beta  = dw - ydz  - \sum_{i=2}^n y_i dx_i,
\]
where $\zeta=(w,y,x_2,y_2,\ldots,x_n,y_n)$ are fibre coordinates;
see \cite[Theorem 1.1]{AlarconForstneric2019IMRN}. Here we are working with a contact 1-form 
of finite degree of smoothness and each step of the normalisation procedure drops the degree by one, 
so we stop after the second step in order to only lose two derivatives.
This suffice for the proof of Theorem \ref{th:Mergelyan}.
\qed\end{remark}

\begin{proof}
Let $p$ denote a point in $S$, and let $\zeta=(\zeta_1,\ldots,\zeta_{2n})$ be complex 
coordinates on $\C^{2n}$, called fibre coordinates. 
Along the $\beta$-Legendrian curve $S\times \{0\}^{2n}=\{\zeta=0\}$ we have 
\[
	\beta(p,0)=\sum_{j=1}^{2n} a_j(p) d\zeta_j,\quad\  p\in S
\]
for some functions $a_j\in\Ascr^{r+1}(S)$ without common zeros. The $1$-form $dz$ 
does not appear in the above expression since $S\times \{0\}^{2n}$ 
is a $\beta$-Legendrian curve.

Let $a=(a_1,\ldots,a_{2n}): S \to \C^{2n}\setminus \{0\}$. Lemma \ref{lem:GL} furnishes a map
$B : S \to GL_{2n}(\C)$ of class $\Ascr^{r+1}(S)$ satisfying $a(p)\cdotp B(p)=(1,0,\ldots,0)$ 
for all $p\in S$. We introduce new fibre coordinates 
\[
	\zeta'=B(p)^{-1}\zeta,\quad\  p\in S,\ \zeta\in\C^{2n}.
\]
This transforms $\beta$ along $S\times \{0\}^{2n}$ to the constant $1$-form $d\zeta'_1$. 
Dropping the primes and denoting the variable $\zeta_1$ by $w$, we thus have  
\[
	\beta=dw \ \ \text{along}\ \ S\times \{0\}^{2n}.
\]
This change of coordinates is of class $\Ascr^{r+1}$, but the differentials
of the component functions $b_{i,j}(p)$ of the matrix $B(p)$ contribute terms
$db_{i,j}=b'_{i,j}dz$ with $b'_{i,j}\in \Ascr^r(S)$, so we lose one derivative and the new
1-form $\beta$ is of class $\Ascr^r$. 

Since the coefficient of $dw$ in $\beta$ equals $1$ on $S\times \{0\}^{2n}$, it is a nowhere vanishing 
function $h\in \Ascr^{r}(S\times \rho \B^{2n})$ for some $\rho>0$, and we have that
\[ 
	\frac{1}{h}\beta=dw+ \biggl(\sum_{j=2}^{2n} b_j \zeta_j \biggr) dz 
	+ \sum_{j,k=2}^{2n} c_{j,k}\zeta_k\, d\zeta_j + \wt \beta,
\] 
where $b_j,c_{j,k}\in \Ascr^{r}(S)$ and the remainder $\wt \beta$ contains terms of the type
described in the lemma. We claim that the functions $b_2,\ldots,b_{2n}$ in the coefficient of $dz$ 
have no common zeros on $S$. Indeed, since $h=1$ on $S\times \{0\}^{2n}$, 
at a common zero $p_0\in S$ of these functions the $2$-form $d\beta$ at the point $(p_0,0)$
does not contain the term $dz$ and hence $\beta \wedge(d\beta)^n$ vanishes, a contradiction. 
Write $\zeta'=(\zeta_2,\ldots,\zeta_{2n})$. Lemma \ref{lem:GL} applied to the map 
$(b_2,\ldots,b_{2n}) : S\to \C^{2n-1} \setminus\{0\}$ gives a change of coordinates 
\[
	(p,w,\zeta')\mapsto (p,w,\wt B(p)\zeta'), \quad\  \wt B(p)\in GL_{2n-1}(\C)
\]
of class $\Ascr^{r}(S)$ such that the coefficient of $dz$ becomes $-\zeta_2$, and hence 
\begin{equation}\label{eq:Taylor2}
	\frac{1}{h} \beta= dw - \zeta_2 dz +  \sum_{j,k=2}^{2n} c_{j,k}\zeta_k \, d\zeta_j + \tilde \beta
\end{equation}
for some new functions $c_{j,k}\in \Ascr^{r}(S)$. Renaming the variable $\zeta_2$ to $y$
and moving the terms  in \eqref{eq:Taylor2} containing $\zeta_j d\zeta_2$ or $\zeta_2d\zeta_j$ with $j> 2$
to the remainder $\wt \beta$ gives \eqref{eq:beta-normal}.
Note that $\wt \beta$ is now of class $\Ascr^{r-1}(S\times\C^{2n})$ since the last change of 
coordinates contributes terms coming from the differentials of the components of $\wt B(p)$.
\end{proof}

%
%
%

Precomposing the immersion $F:S\times \B^{2n}\to X$ by the change of coordinates 
furnished by Lemma \ref{lem:NF} gives a new immersion $F:S\times \rho\B^{2n}\to X$ of class
$\Ascr^r(S\times \rho\B^{2n}, X)$ for some $0<\rho<1$ which agrees with $f$ on 
$S\times \{0\}^{2n}$ such that the contact $1$-form $\frac{1}{h} F^*\eta$ of class $\Cscr^{r-1}$
is asymptotically holomorphic to order $r-1$ on $S\times \rho\B^{2n}$ and of
the form \eqref{eq:beta-normal}. Under this assumption we now prove Theorem \ref{th:Mergelyan}.

%
%
%
%
\begin{proof}[Proof of Theorem \ref{th:Mergelyan}]
We first prove the theorem in the special case when $S$ is a 
compact connected domain in $M$ (i.e., there are no attached arcs)
and without the interpolation conditions. The general case will be considered afterwards.

Let $\Ccal=\{C_1,\ldots, C_l\}$ be a homology basis of $S$ furnished by 
Lemma \ref{lem:homology-basis-admissible}, consisting of piecewise 
smooth oriented Jordan curves with a common base point $p_0\in \mathring S$,
such that $C=\bigcup_{i=1}^l C_i$ is $\Oscr(S)$-convex and each curve $C_i\in \Ccal$ 
contains a nontrivial arc $I_i$ disjoint from $\bigcup_{j\ne i} C_j$.
Let $z:M\to\C$ be the holomorphic immersion chosen at the beginning of the section such that 
\eqref{eq:beta-normal} holds. As in \cite[Sect.\ 4]{AlarconForstnericLopez2017CM} we find a 
holomorphic spray of functions 
\begin{equation}\label{eq:y1}
	y(p,t)= \sum_{i=1}^l t_i \xi_i(p),\quad\  p\in S,\ t=(t_1,\ldots,t_l)\in\C^l,
\end{equation}
where $\xi_i\in \Oscr(S)$ are holomorphic functions satisfying
\begin{equation}\label{eq:periods}
	\int_{C_i} \xi_j dz = \delta_{i,j}, \quad\  i,j=1,\ldots,l. 
\end{equation}
(Here, $\delta_{i,j}$ denotes the Kronecker delta.) Inserting the values
\begin{equation}\label{eq:specialvalues}
	y=y(p,t),\quad \zeta_i=0 \ \ \text{for\ $i=3,\ldots,2n$} 
\end{equation}
into the $1$-form $\alpha=dw-ydz$ (see \eqref{eq:beta-normal}) gives the equation 
\begin{equation}\label{eq:dwisydz}
	dw = y(p,t) dz,\quad\ p\in S
\end{equation}
whose solutions $w=w(p,t)$ are $\alpha$-Legendrian curves. Since the variable $w$ does
not appear on the right hand side, the solutions are obtained by integration:
\begin{equation}\label{eq:wpt}
	w(p,t)= w_0 + \int_{p_0}^p y(\cdotp,t)dz 
	= w_0 + \sum_{i=1}^l t_i \int_{p_0}^p \xi_i dz, \quad\ p\in S.
\end{equation}
From \eqref{eq:y1} and \eqref{eq:dwisydz} it follows that any solution satisfying the initial 
condition $w(p_0,t)=0$ also satisfies an estimate
\begin{equation}\label{eq:estimate-w}
	|w(p,t)| = O(|t|), \quad\  p\in S,
\end{equation}
provided that in \eqref{eq:wpt} we integrate along an approximately geodesic curve in $S$ 
from $p_0$ to $p$. (The integral may of course depend on the choice of the curve
due to nontrivial periods.)
Using the notation \eqref{eq:period}, we see from \eqref{eq:periods} that the period map
of the solution \eqref{eq:wpt} along the curve $C_i$ equals
\[
	\Pscr_{C_i}(p_0,w_0,t) = t_i,\quad\  i=1,\ldots,l.
\]
Hence, the period map $\C^l\ni t\mapsto \Pscr^\alpha(t)\in \C^l$ of the family of solutions \eqref{eq:wpt} 
associated to the 1-form $\alpha=dw-ydz$ and the spray \eqref{eq:y1} is the identity map 
\[
	\Pscr^\alpha(t)=\left(\Pscr^\alpha_{C_1}(t),\ldots,\Pscr^\alpha_{C_l}(t)\right) =t.
\] 
In particular, the only single-valued $\alpha$-Legendrian curve in 
this family satisfying the initial condition $w(p_0,t)=0$ is $w=0$ for the parameter value $t=0$. 

Inserting the values \eqref{eq:specialvalues} into the $1$-form $\frac{1}{h}\beta$ 
\eqref{eq:beta-normal}, the second term on the right-hand side vanishes,
while Lemma \ref{lem:NF} shows that the only nonvanishing terms in the remainder $\wt \beta$ 
are those of the form $wdy$ and $ydy$, possibly multiplied by other normal coordinates and by 
functions in $\Ascr^{r-1}(S)$. We see from \eqref{eq:y1} and \eqref{eq:estimate-w} that these
terms disturb the period map of solutions of the resulting differential equation 
for $\beta$-Legendrian curves only by a term of size $O(|t|^2)$. Hence, the period map of 
the $\beta$-Legendrian curve satisfying the initial condition $w(p_0,t)=0$  equals
\[
	\Pscr^{\beta}(t)=t+O(|t|^2).
\] 
(In fact, a weaker estimate $\Pscr^{\beta}(t)=t+o(|t|)$ suffices in the sequel.)

For every small $\delta>0$ the map $\Pscr^{\beta}(t)$ is close enough to $t$ on the closed polydisc
\begin{equation}\label{eq:polydisc}
	P_\delta =\{(t_1,\ldots,t_l):|t_i|\le \delta,\  i=1,\ldots,l\} = \delta \overline\D^{l}\subset \C^l
\end{equation}
that it maps $bP_\delta$ to $\C^l_*$ and this map has degree one. 
(Since $bP_\delta$ is homotopic to the sphere $S^{2l-1}$ in $\C^l_*$ and $\C^l_*$ 
retracts onto $S^{2l-1}$, the degree is well defined, cf.\ Subsection \ref{ss:degree}.)

We now fix $\delta>0$; however, its value will be determined only later.

Let $F:S\times \rho\B^{2n}\to X$ be the immersion constructed above 
so that $F(\cdotp,0)=f$ and $\frac{1}{h}F^*\eta=\frac{1}{h}\beta$  
is of the form \eqref{eq:beta-normal}. After decreasing $\rho>0$ slightly, 
we can approximate $F$ as closely as desired in the $\Cscr^r(S\times \rho\B^{2n},X)$ 
topology by a holomorphic immersion $G:U\times \rho\B^{2n}\to X$, where $U\subset M$ 
is a neighbourhood of $S$. Let us explain this.
By the construction, the graph of $F$ is contained in a Stein domain $\wt\Omega\subset M\times X$. 
Hence, using a holomorphic embedding of $\wt\Omega$ into a Euclidean space $\C^N$ 
and an ambient holomorphic retraction back to this embedded submanifold, 
the proof reduces to approximation of functions in $\Ascr^r(S\times \rho\B^{2n})$ 
by holomorphic functions in a neighbourhood (decreasing $\rho>0$ slightly). 
For the details, see \cite[Sect.\ 7.2, Lemma 3]{FornaessForstnericWold2018}. 
In order to approximate a function in $\Ascr^r(S\times \rho\B^{2n})$, we consider 
its Taylor series expansion in the fibre variables, 
\[
	(p,z) \mapsto \sum_{I\in \Z_+^{2n}} a_I(p) z^I 
\]
with coefficients $a_I\in \Ascr^r(S)$. It remains to approximate the coefficients $a_I$
in the  $\Ascr^r(S)$ by functions $\tilde a_I\in\Oscr(S)$; this is accomplished by
\cite[Theorem 16]{FornaessForstnericWold2018}.

Suppose now that $G:U\times \rho\B^{2n}\to X$ is a holomorphic immersion
approximating $F$ in $\Cscr^r(S\times \rho\B^{2n},X)$. Recall that $\eta$ is the contact form in $X$. The pullback 
\begin{equation}\label{eq:tildebeta}
	  \gamma  := G^*\eta 
\end{equation}
is then a holomorphic contact form on $U\times \rho \B^{2n}$ which is $\Cscr^{r-1}$-close to $\beta$
on $S\times \rho\B^{2n}$. Furthermore, the coefficient $g \in\Oscr(U\times \rho \B^{2n})$
of the differential $dw$ in $\gamma$ is close to the corresponding coefficient $h$ of $\beta$
on $S\times \rho\B^{2n}$ and hence is nonvanishing, perhaps after shrinking $U\supset S$ 
and decreasing $\rho>0$ slightly. The holomorphic contact form $\frac{1}{g}\gamma$
on $U\times \rho \B^{2n}$ is then $\Cscr^{r-1}$ close to the $1$-form $\frac{1}{h}\beta$ 
\eqref{eq:beta-normal} on $S\times  \rho \B^{2n}$.

We now insert the values \eqref{eq:specialvalues} into the contact form $\frac{1}{g}\gamma$ and denote 
by $t\mapsto \Pscr^{\gamma}(t)$ the corresponding period map for solutions satisfying 
the initial condition $w(p_0,t)=0$. Assuming that the approximations are close enough, 
the period map $\Pscr^{\gamma}(t)$ is so close to $\Pscr^{\beta}(t)=t+O(|t|^2)$ uniformly
on the polydisc $P_\delta\subset\C^l$ (see \eqref{eq:polydisc}) that it maps 
$bP_\delta$ to $\C^l_*$ and this map has degree one. By Proposition \ref{prop:degree}
there is a point $t^0\in \mathring P_\delta$ such that 
\[
	\Pscr^{\gamma}(t^0)=0. 
\]
For $t=t^0$, the solution of the differential equation for $\gamma$-Legendrian curves 
satisfying the initial condition $w(p_0,t^0)=0$ has vanishing periods over the curves $C_1,\ldots, C_l$ 
in the homology basis of $S$. Assuming that the number $\delta>0$ was chosen small enough
and the approximations were close enough, we obtain an embedded holomorphic 
$\gamma$-Legendrian curve on a neighbourhood of $S$ in $M$ (see Subsect.\ \ref{ss:ODE}) 
which is $\Cscr^r$-close to the initial $\beta$-Legendrian curve $S\times \{0\}^{2n}$.
(Note that one derivative is gained when integrating the differential equation.)
Its image by $G$ is a holomorphic $\xi$-Legendrian immersion $\tilde f:U\to X$ 
which approximates the initial Legendrian immersion $f:S\to X$ in $\Cscr^r(S,X)$. 

This completes the proof of Theorem \ref{th:Mergelyan} in the special case.

%
%
Consider now the general case when $S=K\cup E$ is an admissible set (see Definition \ref{def:admissible}). 
Without a loss of generality we shall assume that $S$ is connected.
Let $K=\bigcup_{i=1}^m K_i$, where $K_1,\ldots,K_m$ are the connected components of $K$. 
Let $\Ccal$ denote a homology basis of $S$ furnished by Lemma \ref{lem:homology-basis-admissible}. 
Recall that all curves $C\in \Ccal$ are based at the vertex $q_1\in K_1$. 
We enlarge the finite set $A\subset S$ (at which we shall interpolate)
by adding to it the endpoints of all arcs $E_i \subset E$ (the connected components of $E$) and
the endpoints of all arcs $A_{i,j}$ chosen in the proof of Lemma \ref{lem:homology-basis-admissible}.
(The latter set includes all vertices $q_i\in K_i$.)

We then form a family $\Cscr$ of arcs and closed curves in $S$ as follows. 

\begin{enumerate} [\rm (a)]
\item 
If a curve $C\in \Ccal$ does not contain any points of $A$ except $q_1$, we put it in $\Cscr$.
Otherwise, we split $C$ into the union of finitely many arcs lying back to back, 
with the points of $A\cap C$ as the common endpoints of adjacent arcs, and 
we put all these arcs in $\Cscr$.  
\smallskip
\item 
If $E_k$ is a connected component of $E$ which is not contained in any 
of the curves from the previous item, we connect each
endpoint of $E_k$ contained in a connected component $K_{i}$ of $K$ to the vertex
$q_i$ as described in Case 2 in the proof  of Lemma \ref{lem:homology-basis-admissible}
(first going to a suitable point $a_{i,j}$ and then going along $A_{i,j}$ to $q_i$).
We then split the resulting curve into arcs at the points of $A$ as in the previous case
and put all these arcs in $\Cscr$. 

\smallskip
\item 
Let $A'$ denote the set of points $a\in A$ belonging to at least one curve in the family $\Cscr$ 
constructed thus far. Any remaining point $a\in A\setminus A'$ lies in one of the components 
$K_i$ of $K$. Choose an embedded arc $\Lambda_a \subset K_i$ 
connecting $a$ to the vertex $q_i$ such that $\Lambda_a$
does not meet any of the arcs from $\Cscr$ other than at $q_i$. 
We put the arcs $\Lambda_a$ for $a\in A\setminus A'$ 
in the family $\Cscr$ whose construction is now complete.
\end{enumerate}

Note that the union $\wt C$ of all curves in the family $\Cscr$ is a connected Runge set in $S$.

Let $y(p,t)$ be a spray \eqref{eq:y1} where the functions $\xi_i\in \Oscr(S)$ 
satisfy conditions \eqref{eq:periods} on the curves in the family $\Cscr$.
As in the special case considered above, we approximate $F$ in the $\Cscr^r(S\times \C^{2n})$
topology by a holomorphic immersion $G:U\times \rho\B^{2n}\to X$ from a neighbourhood 
of $S\times \{0\}^{2n}$ into $X$ which agrees with $F$ at the finitely many points $A\times \{0\}^{2n}$,
and let $\gamma=G^*\eta$ \eqref{eq:tildebeta}. Let 
$g\in\Oscr(U\times \rho \B^{2n})$ denote the nonvanishing coefficient of $dw$ in $\gamma$.
We insert the values $y=y(p,t)$ and $\zeta_i=0$ for $i=3,\ldots,2n$ (see \eqref{eq:specialvalues})
into the equation $\frac{1}{g}\gamma=0$ for $\gamma$-Legendrian curves.
Set $p_0:=q_1\in A$ (recall that $q_1$ is the vertex of $K_1$). 
Assume that $\Cscr$ contains $l$ curves (removing the repetitions).
For each $C\in\Cscr$ choose a regular parameterisation by a function 
$\lambda(s)$, $s\in [0,1]$. For $t\in\C^l$ sufficiently close to $0$ we define
\[
	\Pscr_C^{\gamma}(t) = w(1,t) \in \C,
\] 
where $[0,1]\ni s\mapsto w(s,t)$ is the unique solution on $C$ of the differential equation for 
$\gamma$-Legendrian curves with the initial value $w(0,t)=0$. 
(Compare with \eqref{eq:Jordancurve} and \eqref{eq:period}.)
For a small enough $\delta>0$ this defines the period map $t\mapsto \Pscr^{\gamma}(t)\in \C^l$ 
associated to $\Cscr$, with $t\in P_\delta\subset \C^l$ (see \eqref{eq:polydisc}).
Assuming as we may that the approximations were close enough, the same argument 
as in the special case considered above gives a value $t^0\in \mathring P_\delta$ such that 
$\Pscr^{\gamma}(t_0)=0$. Since the union $\wt C$ of all curves in the family $\Cscr$ is connected, 
it follows that the solution of the differential equation for $\tilde \beta$-Legendrian curves with $t=t^0$ 
and satisfying the initial condition $w(p_0,t^0)=0$ is single-valued on $\wt C$ and it vanishes 
at all points in $A$. Assuming that $\delta>0$ 
was chosen small enough and the approximation of $\beta$ by $\gamma$ was close enough, 
we obtain a single-valued holomorphic solution on a neighbourhood of $S$ in $M$ vanishing
at all points of $A$. (See Subsect.\ \ref{ss:ODE}.)
Its image by $G$ is a holomorphic $\xi$-Legendrian immersion $\tilde f:U\to X$ from a neighbourhood 
of $S$ which approximates the initial Legendrian immersion $f:S\to X$ in $\Cscr^r(S,X)$ 
and agrees with $f$ at the points of $A$. Clearly, this method gives interpolation to any
given finite order at  the points in $A\cap \mathring S$ provided we choose $G$ to match
$F$ to a suitable finite order at the points in $(A\cap \mathring S)\times \{0\}^{2n}$.
The latter condition is a standard addition to the Mergelyan approximation theorem.
\end{proof}

%
%
%
%
\section{Proof of Theorem \ref{th:soft}}\label{sec:soft}

By the general position theorem for Legendrian immersions 
(see \cite[Theorem 1.2]{AlarconForstneric2019IMRN}) and shrinking the open set $U$ around $K$
if necessary, we may assume that $f:U\to X$ is a Legendrian embedding. 

Theorem \ref{th:soft} is proved by inductively applying the Mergelyan
approximation theorem for Legendrian immersions, given by Theorem \ref{th:Mergelyan}, and the
procedure described (for Stein manifolds of any dimension $n\ne 2$) 
in \cite[Proof of Theorem 1.2]{ForstnericSlapar2007MZ}. We provide the outline.

Choose a strongly subharmonic Morse exhaustion function $\rho : M\to \R_+$ and an increasing sequence
$0<c_0<c_1<c_2\cdots$ of regular values of $\rho$ with $\lim_{j\to\infty} c_j=+\infty$
such that, setting $M_j=\{\rho<c_j\}$ for $j\in\Z_+$, we have that
$K\subset M_0\subset \overline M_0\subset U$ and for each $j>0$ the function $\rho$ has 
at most one critical point $p_j$ in $M_j \setminus M_{j-1}$. Fix $\epsilon>0$ and set 
$\epsilon_0=\epsilon/2$, $W_0=M_0$, and $h_0=\Id_M$. 
We inductively construct an increasing sequence of smoothly bounded 
relatively compact domains $W_0\Subset W_1\Subset W_2\Subset \cdots$ in $M$ (not necessarily exhausting $M$),  
a sequence of continuous maps $f_j: M\to X$, a sequence of diffeomorphisms $h_j: M\to M$,
and a decreasing sequence of numbers $\epsilon_j>0$
such that the following conditions hold for all $j=1,2,\ldots$.
\begin{enumerate}[\rm (i)]
\item The compact set $\overline W_{j-1}$ is $\Oscr(W_j)$-convex. 
\item The map $f_j$ is a holomorphic Legendrian embedding on a neighborhood of $\overline W_j$, 
and it is homotopic to $f_{j-1}$ by a homotopy $f_{j,t}: M\to X$ $(t\in [0,1])$ such that each 
$f_{j,t}$ is a holomorphic Legendrian embedding on $\overline W_{j-1}$ satisfying
\[
	\sup_{p\in \overline W_{j-1}} \dist \bigl(f_{j,t}(p),f_{j-1}(p)\bigr) < \epsilon_{j-1}. 
\]
\item $h_j(M_j)=W_j$, and $h_j=h_{j-1}\circ g_j$ where $g_j: M\to M$ is a diffeomorphism which is 
diffeotopic to $\Id_M$ by a diffeotopy that is fixed on a neighborhood of $\overline M_{j-1}$.
\item We have $\epsilon_j < \frac12\min\{\epsilon_{j-1},\delta_j\}$, where $\delta_j>0$ is chosen such that 
any holomorphic map $g: W_{j}\to X$ satisfying 
$\sup_{p\in W_{j}} \bigl(f_j(p),g(p)\bigr) < \delta_j$ is an embedding on $W_{j-1}$.
\end{enumerate}
Granted such sequences, it is easily verified that there exists the limit map 
\[ 
	\tilde f=\lim_{j\to \infty} f_j : \Omega=\bigcup_{j=1}^\infty W_j \to X
\]
which is a holomorphic Legendrian embedding, and there exists the limit diffeomorphism 
$h=\lim_{j\to\infty} h_j : M\to \Omega$ onto $\Omega$. The composition $F=\tilde f\circ h:M\to X$ and the 
complex structure $J=h^*(J_0)$ on $M$ then satisfy Theorem \ref{th:soft}, i.e., 
$F$ is a $J$-holomorphic Legendrian embedding approximating $f$ on $K$ and $J$ agrees with the original
complex structure $J_0$ on a neighborhood of $K$ (since $h$ is the identity there).

The induction begins with $f_0=f$, $W_0=M_0$, $h_0=\Id_M$, and $\epsilon_0=\epsilon/2$.
We now explain the inductive step. Fix $j\in\N$. Assume that 
$\rho$ has a (unique) critical point $p_j \in M_j \setminus \overline M_{j-1}$.
If $p_j$ is a local minimum of $\rho$, we let $E_j=\{p_j\}$. Otherwise,
the Morse index of $p_j$ equals $1$ and the change of topology of the sublevel
set $\{\rho<c\}$ at $p_j$ is described by attaching to the compact domain $\overline M_{j-1}$
a smooth embedded arc $E_j\subset M_j$ intersecting $\overline M_{j-1}$ 
transversely at both endpoints and nowhere else. Finally, if $\rho$ has no critical point in
$M_j \setminus \overline M_{j-1}$, we let $E_j=\varnothing$. In all three cases, the 
compact set
\begin{equation}\label{eq:Gammaj}
	\Gamma_j:=\overline M_{j-1}\cup E_j \subset M_j
\end{equation}
has arbitrarily small smoothly bounded neighbourhoods $M'_j\Subset M_j$ 
which are diffeotopic to $M_j$ by a diffeotopy of $M$ that is fixed on a 
neighbourhood of $\Gamma_j$. 

Let $h_{j-1}:M\to M$ be the diffeomorphism from the
previous step, so we have that $W_{j-1}=h_{j-1}(M_{j-1})$. Set $E'_j=h_{j-1}(E_j)$. Then,
\[
	 S_j:= \overline W_{j-1}\cup E'_j = h_{j-1}(\Gamma_j) 
\] 
is an admissible subset of $M$ (see Definition \ref{def:admissible}). 
By the induction hypothesis, the map $f_{j-1}:M\to X$ is a Legendrian embedding on a neighbourhood of 
$\overline W_{j-1}$. Our goal is to find a homotopic deformation of $f_{j-1}$ to continuous
map $f_j:M\to X$ which is a holomorphic Legendrian embedding on a neighbourhood of $S_j$.

If $E'_j=\varnothing$, there is nothing to do. If $E'_j$ is a point, we let
$f_j$ agree with $f_{j-1}$ near $\overline W_{j-1}$ and let it be an arbitrary 
holomorphic Legendrian embedding on a small neighbourhood of $E'_j$. Thus, the only
nontrivial case is when $E'_j$ is a smooth arc attached transversely to $\overline W_{j-1}$
at its endpoints. In this case we first homotopically deform $f_{j-1}$, keeping it fixed near 
$\overline W_{j-1}$, so that $f_{j-1}:E'_j\to X$ becomes a smoothly immersed Legendrian curve.
This is possible by the Chow-Rashevski\u{\i} theorem; 
see Gromov \cite[1.1, p.\ 113 and 1.2.B, p.\ 120]{Gromov1996}.
This makes $f_{j-1}:S_j\to X$ a smooth Legendrian immersion which is holomorphic 
on a neighbourhood of $\overline W_{j-1}$. By Theorem \ref{th:Mergelyan} 
and the general position theorem (see \cite[Theorem 1.2]{AlarconForstneric2019IMRN})
we can approximate $f_{j-1}$ as closely as desired in the $\Cscr^2$ topology on $S_j$ 
by a holomorphic Legendrian embedding $f_j:U_j\to X$ from a neighborhood
of $S_j$ into $X$. After shrinking $U_j$ around $S_j$ there exists 
a homotopy $f_{j,t}: U_j\to X$ $(t\in [0,1])$ between $f_{j-1}|_{U_j}=f_{j,0}$ and $f_j$ 
consisting of maps which are Legendrian embeddings on a neighbourhood $V_j\subset U_j$ 
of $\overline W_{j-1}$ (see \cite[Remark 3.2]{AlarconForstneric2019IMRN}).
Assuming that the approximations were close enough we get condition (ii).

By what was said above, there is a smoothly bounded neighbourhood $W_j\Subset U_j$ of $S_j$ 
of the form $W_j=h_{j-1}(M'_j)$, where $M'_j\subset M_j$ is a neighbourhood of 
$\Gamma_j$ \eqref{eq:Gammaj} diffeotopic to $M_j$ by a diffeotopy 
which is fixed on a neighbourhood of $\Gamma_j$. Let $g_{j,t}:M\to M$ $(t\in [0,1])$ be such a diffeotopy
with $g_{j,0}=\Id_M$ and $g_{j,1}(M_j)=M'_j$. Then, 
\[	
	h_{j,t}:=h_{j-1}\circ g_{j,t}:M\to M,\quad\ t\in [0,1], 
\]
is a diffeotopy connecting $h_{j,0}=h_{j-1}$ and $h_{j,1}=h_j$ such that $h_j(M_j)=W_j$. This gives condition (iii). 
By using a cutoff function in the parameter of the homotopy we can extend $f_j$ and the homotopy $f_{j,t}$
(keeping it fixed on a neighbourhood of $S_j$)  to a continuous map $M\to X$ which agrees with $f_{j-1}$ on 
$M\setminus U_j$. Finally, we choose the next number $\epsilon_j>0$ sufficiently small so that
condition (iv) holds. This concludes the proof of Theorem \ref{th:soft}.


\subsection*{Acknowledgements}
Research on this paper was supported by research program P1-0291 and grant J1-9104
from ARRS, Republic of Slovenia, and the Stefan Bergman Prize 2019.

I wish to thank Roger Casals and Nicola Pia for their helpful explanation
and references concerning local stability and functional dimension of Pfaffian systems, described in the 
introduction just after Problem \ref{prob:realanalytic} (personal communications, November 2019). 
I also thank Antonio Alarc\'on for his remarks which helped me to improve the presentation,
and for having contributed Figure \ref{fig:bridges}.




\vspace*{5mm}
\noindent Franc Forstneri\v c

\noindent Faculty of Mathematics and Physics, University of Ljubljana, Jadranska 19, SI--1000 Ljubljana, Slovenia

\noindent Institute of Mathematics, Physics and Mechanics, Jadranska 19, SI--1000 Ljubljana, Slovenia

\noindent e-mail: {\tt franc.forstneric@fmf.uni-lj.si}

\end{document}